\documentclass{amsart}
\usepackage[all]{xy}
\usepackage{verbatim}
\usepackage{color}
\usepackage{amsthm}
\usepackage{amssymb}
\usepackage[colorlinks=true]{hyperref}
\usepackage{footmisc}
\usepackage{stmaryrd}

\numberwithin{equation}{section}

\newtheorem{theorem}[equation]{Theorem}
\newtheorem*{theorem*}{Theorem} \newtheorem{lemma}[equation]{Lemma}

\newtheorem*{conjecture*}{Mamma Conjecture}
\newtheorem*{conjecture1*}{Mamma Conjecture (revisited)}
\newtheorem{proposition}[equation]{Proposition}
\newtheorem{corollary}[equation]{Corollary}
\newtheorem*{corollary*}{Corollary}

\theoremstyle{remark}
\newtheorem{definition}[equation]{Definition}

\newtheorem{example}[equation]{Example}

\newtheorem{notation}[equation]{Notation}

\theoremstyle{remark}
\newtheorem{remark}[equation]{Remark}

\newcommand{\cA}{{\mathcal A}}
\newcommand{\cB}{{\mathcal B}}
\newcommand{\cC}{{\mathcal C}}
\newcommand{\cD}{{\mathcal D}}

\newcommand{\cL}{{\mathcal L}}

\newcommand{\cN}{{\mathcal N}}
\newcommand{\cO}{{\mathcal O}}

\newcommand{\cT}{{\mathcal T}}

\newcommand{\cZ}{{\mathcal Z}}

\newcommand{\bbC}{\mathbb{C}}

\newcommand{\bbF}{\mathbb{F}}
\newcommand{\bbG}{\mathbb{G}}

\newcommand{\bbP}{\mathbb{P}}

\newcommand{\bbQ}{\mathbb{Q}}
\newcommand{\bbZ}{\mathbb{Z}}

\DeclareMathOperator{\id}{id}

\DeclareMathOperator{\NChow}{NChow} 
\DeclareMathOperator{\NHom}{NHom} 
\DeclareMathOperator{\NNum}{NNum} 

\DeclareMathOperator{\Num}{Num} 

\newcommand{\dgcat}{\mathrm{dgcat}} 

\newcommand{\perf}{\mathrm{perf}}

\newcommand{\Chow}{\mathrm{Chow}}

\newcommand{\dg}{\mathrm{dg}}

\newcommand{\Hom}{\mathrm{Hom}}
\newcommand{\End}{\mathrm{End}}

\newcommand{\rep}{\mathrm{rep}}

\newcommand{\dgHo}{\mathrm{H}^0}

\newcommand{\Hmo}{\mathrm{Hmo}}
\newcommand{\op}{\mathrm{op}}

\newcommand{\too}{\longrightarrow}

\newcommand{\ie}{\textsl{i.e.,}\ }

\let\oldmarginpar\marginpar
\def\marginpar#1{\oldmarginpar{\tiny #1}}

\begin{document}

\title[NC motives in positive characteristic and their applications]{Noncommutative motives in positive characteristic and their applications}
\author{Gon{\c c}alo~Tabuada}
\address{Gon{\c c}alo Tabuada, Department of Mathematics, MIT, Cambridge, MA 02139, USA}
\email{tabuada@math.mit.edu}
\urladdr{http://math.mit.edu/~tabuada}
\thanks{The author was partially supported by a NSF CAREER Award}

\date{\today}
\abstract{Let $k$ be a base field of positive characteristic. Making use of topological periodic cyclic homology, we start by proving that the category of noncommutative numerical motives over $k$ is abelian semi-simple, as conjectured by Kontsevich. Then, we establish a far-reaching noncommutative generalization of the Weil conjectures, originally proved by Dwork and Grothendieck. In the same vein, we establish a far-reaching noncommutative generalization of the cohomological interpretations of the Hasse-Weil zeta function, originally proven by Hesselholt. As a third main result, we prove that the numerical Grothendieck group of every smooth proper dg category is a finitely generated free abelian group, as claimed (without proof) by Kuznetsov. Then, we introduce the noncommutative motivic Galois (super-)groups and, following an insight of Kontsevich, relate them to their classical commutative counterparts. Finally, we explain how the motivic measure induced by Berthelot's rigid cohomology can be recovered from the theory of noncommutative motives.}}
\maketitle
\tableofcontents



\section{Preliminaries}
Throughout the article $k$ denotes a base field of positive characteristic $p>0$. When $k$ is perfect, we will write $W(k)$ for the ring of $p$-typical Witt vectors of $k$ and $K:=W(k)_{1/p}$ for the fraction field of $W(k)$.
\subsection{Dg categories}\label{sub:dg}
Let $\cC(k)$ the category of complexes of $k$-vector spaces. A {\em differential
  graded (=dg) category $\cA$} is a category enriched over $\cC(k)$ and a {\em dg functor} $F\colon\cA\to \cB$ is a functor enriched over
$\cC(k)$; consult Keller's survey
\cite{ICM-Keller}. In what follows, we will write $\dgcat(k)$ for the category of small dg categories.

Every (dg) $k$-algebra $A$ gives naturally rise to a dg category with a single object. Another source of examples is proved by $k$-schemes (or, more generally, by algebraic $k$-stacks) since the category of perfect complexes $\perf(X)$ of every $k$-scheme $X$ admits a canonical dg enhancement\footnote{When $X$ is quasi-projective this dg enhancement is unique; see Lunts-Orlov \cite[Thm. 2.12]{LO}.} $\perf_\dg(X)$; consult \cite[\S4.6]{ICM-Keller}.

Let $\cA$ be a dg category. The opposite dg category $\cA^\op$, resp. category $\dgHo(\cA)$, has the same objects as $\cA$ and $\cA^\op(x,y):=\cA(y,x)$, resp. $\dgHo(\cA)(x,y):=H^0\cA(x,y)$. A {\em right dg
  $\cA$-module} is a dg functor $M\colon \cA^\op \to \cC_\dg(k)$ with values
in the dg category $\cC_\dg(k)$ of complexes of $k$-vector spaces. Let
us write $\cC(\cA)$ for the category of right dg
$\cA$-modules. Following \cite[\S3.2]{ICM-Keller}, the derived
category $\cD(\cA)$ of $\cA$ is defined as the localization of
$\cC(\cA)$ with respect to the objectwise quasi-isomorphisms. In what follows, we will write
$\cD_c(\cA)$ for the triangulated subcategory of compact objects.

A dg functor $F\colon\cA\to \cB$ is called a {\em Morita equivalence} if it induces an equivalence between derived categories $\cD(\cA) \simeq
\cD(\cB)$; see \cite[\S4.6]{ICM-Keller}. As explained in
\cite[\S1.6]{book}, the category $\dgcat(k)$ admits a Quillen model
structure whose weak equivalences are the Morita equivalences. Let $\Hmo(k)$ be the associated homotopy category.

The {\em tensor product $\cA\otimes\cB$} of dg categories is defined
as follows: the set of objects is the cartesian product and
$(\cA\otimes\cB)((x,w),(y,z)):= \cA(x,y) \otimes\cB(w,z)$. As
explained in \cite[\S2.3]{ICM-Keller}, this construction gives rise to
a symmetric monoidal structure $-\otimes -$ on $\dgcat(k)$, which descends to the homotopy category
$\Hmo(k)$. 

A {\em dg $\cA\text{-}\cB$-bimodule} is a dg functor
$\mathrm{B}\colon \cA\otimes \cB^\op \to \cC_\dg(k)$ or, equivalently, a
right dg $(\cA^\op \otimes \cB)$-module. A standard example is the dg
$\cA\text{-}\cB$-bimodule
\begin{eqnarray}\label{eq:bimodule2}
{}_F\mathrm{B}\colon\cA\otimes \cB^\op \too \cC_\dg(k) && (x,z) \mapsto \cB(z,F(x))
\end{eqnarray}
associated to a dg functor $F\colon\cA\to \cB$. Let us write $\rep(\cA,\cB)$ for the full triangulated subcategory of $\cD(\cA^\op \otimes \cB)$ consisting of those dg $\cA\text{-}\cB$-modules $\mathrm{B}$ such that for every object $x \in \cA$ the associated right dg $\cB$-module $\mathrm{B}(x,-)$ belongs to $\cD_c(\cB)$.

Finally, recall from Kontsevich \cite{Miami,finMot,IAS} that a dg category $\cA$ is called {\em smooth} if the dg $\cA\text{-}\cA$-bimodule ${}_{\id}\mathrm{B}$ belongs to the subcategory $\cD_c(\cA^\op\otimes \cA)$ and {\em proper} if $\sum_i \mathrm{dim}\, H^i\cA(x,y)< \infty$ for any ordered pair of objects $(x,y)$. Examples include the finite-dimensional $k$-algebras of finite global dimension $A$ (when $k$ is perfect) as well as the dg categories $\perf_\dg(X)$ associated to smooth proper $k$-schemes $X$ (or, more generally, to smooth proper algebraic $k$-stacks). As proved in \cite[Thm.~1.43]{book}, the smooth proper dg categories can be (conceptually) characterized as the dualizable objects of the symmetric monoidal category $\Hmo(k)$. Moreover, the dual of a smooth proper dg category $\cA$ is its opposite dg category $\cA^\op$.
\subsection{Noncommutative motives}\label{sec:NCmotives}
For a book, resp. survey, on noncommutative motives, we invite the reader to consult \cite{book}, resp. \cite{survey}.

As explained in \cite[\S1.6.3]{book}, given dg categories $\cA$ and $\cB$, there is a natural bijection between $\Hom_{\Hmo(k)}(\cA,\cB)$ and the set of isomorphism classes of the category $\rep(\cA,\cB)$. Moreover, under this bijection, the composition law of $\Hmo(k)$ corresponds to the derived tensor product of bimodules. Since the dg $\cA\text{-}\cB$ bimodules \eqref{eq:bimodule2} belong to $\rep(\cA,\cB)$, we hence have the following symmetric monoidal functor:
\begin{eqnarray}\label{eq:func1}
\dgcat(k) \too \Hmo(k) & \cA\mapsto \cA & F \mapsto {}_F \mathrm{B}\,.
\end{eqnarray}
The {\em additivization} of $\Hmo(k)$ is the additive category $\Hmo_0(k)$ with the same objects as $\Hmo(k)$ and with morphisms given by $\Hom_{\Hmo_0(k)}(\cA,\cB):=K_0\rep(\cA,\cB)$, where $K_0\rep(\cA,\cB)$ stands for the Grothendieck group of the triangulated category $\rep(\cA,\cB)$. The composition law and the symmetric monoidal structure are induced from $\Hmo(k)$. Consequently, we have the following symmetric monoidal functor:
\begin{eqnarray}\label{eq:func2}
\Hmo(k) \too \Hmo_0(k) & \cA \mapsto \cA & \mathrm{B} \mapsto [\mathrm{B}]\,.
\end{eqnarray}
Given a commutative ring of coefficients $R$, the {\em $R$-linearization} of $\Hmo_0(k)$ is the $R$-linear category $\Hmo_0(k)_R$ obtained by tensoring the morphisms of $\Hmo_0(k)$ with $R$. By construction, we have the symmetric monoidal functor:
\begin{eqnarray}\label{eq:func3}
\Hmo_0(k) \too \Hmo_0(k)_R & \cA \mapsto \cA & [\mathrm{B}] \mapsto [\mathrm{B}]_R\,.
\end{eqnarray}
Let us write $U(-)_R\colon \dgcat(k) \to \Hmo_0(k)_R$ for the composition $\eqref{eq:func3}\circ \eqref{eq:func2}\circ \eqref{eq:func1}$.
\subsubsection{Noncommutative Chow motives}\label{sub:Chow}
The category of {\em noncommutative Chow motives} $\NChow(k)_R$ is defined as the idempotent completion of the full subcategory of $\Hmo_0(k)_R$ consisting of the objects $U(\cA)_R$ with $\cA$ a smooth proper dg category. By construction, this category is $R$-linear, additive, idempotent complete, and symmetric monoidal. Since the smooth proper dg categories are the dualizable objects of the symmetric monoidal category $\Hmo(k)$, the category $\NChow(k)_R$ is moreover {\em rigid}, \ie all its objects are dualizable. Let us denote by $(-)^\vee$ the associated (contravariant) duality functor. Given dg categories $\cA$ and $\cB$, with $\cA$ smooth proper, we have an equivalence of categories $\rep(\cA,\cB) \simeq \cD_c(\cA^\op \otimes \cB)$; see \cite[Cor.~1.44]{book}. Consequently, since $K_0(\cA^\op \otimes \cB):=K_0(\cD_c(\cA^\op \otimes \cB))$, we obtain the isomorphisms:
\begin{equation}\label{eq:Homs}
\Hom_{\NChow(k)_R}(U(\cA)_R,U(\cB)_R):=K_0(\rep(\cA,\cB))_R\simeq K_0(\cA^\op \otimes\cB)_R\,.
\end{equation}
When $R=\bbZ$, we will simply write $\NChow(k)$ instead of $\NChow(k)_\bbZ$.
\subsubsection{Noncommutative numerical motives}\label{sub:numerical}
Given an $R$-linear, additive, rigid symmetric monoidal category $(\mathrm{D}, \otimes, {\bf 1})$, its {\em $\cN$-ideal} is defined as follows
$$ \cN(a,b):=\{f \in \Hom_{\mathrm{D}}(a,b)\,|\, \forall g \in \Hom_{\mathrm{D}}(b,a)\,\,\mathrm{we}\,\,\mathrm{have}\,\,\mathrm{tr}(g\circ f)=0 \}\,,$$
where $\mathrm{tr}(g\circ f)$ stands for the categorical trace of $g\circ f$. The~category of {\em noncommutative numerical motives} $\NNum(k)_R$ is defined\footnote{Consult Remark \ref{rk:Kontsevich} below for an alternative definition of the category $\NNum(k)_R$.} as the idempotent completion of the quotient of $\NChow(k)_R$ by the $\otimes$-ideal $\cN$. By construction, $\NNum(k)_R$ is $R$-linear, additive, idempotent complete, and rigid symmetric~monoidal. 

\subsection{Orbit categories}\label{sub:orbit}
Let $(\mathrm{D}, \otimes, {\bf 1})$ be a $R$-linear, additive, rigid symmetric monoidal category equipped with a $\otimes$-invertible object $\cO \in \mathrm{D}$. The {\em orbit category} $\mathrm{D}/_{\!\!-\otimes \cO}$ has the same objects as $\mathrm{D}$ and $\Hom_{\mathrm{D}/_{\!\!-\otimes \cO}}(a,b):=\bigoplus_{n \in \bbZ} \Hom_{\mathrm{D}}(a,b\otimes \cO^{\otimes n})$. 
Given objects $a$, $b$, and $c$, and morphisms
\begin{eqnarray*}
\{f_n\}_{n \in \bbZ} \in \bigoplus_{n \in \bbZ} \Hom_{\mathrm{D}}(a,b\otimes \cO^{\otimes n}) && \{g_n\}_{n \in \bbZ} \in \bigoplus_{n \in \bbZ} \Hom_{\mathrm{D}}(b,c\otimes \cO^{\otimes n})\,,
\end{eqnarray*}
the $i^{\mathrm{th}}$ component of $\{g_n\} \circ \{f_n\}$ is defined as $\sum_n ((g_{i-n}\otimes \cO^{\otimes n})\circ f_n)$. The functor
\begin{eqnarray}\label{eq:functor}
\gamma\colon \mathrm{D} \too \mathrm{D}/_{\!\!-\otimes \cO} & a \mapsto a & f \mapsto \{f_n\}_{n \in \bbZ}\,,
\end{eqnarray} 
where $f_0=f$ and $f_n=0$ if $n\neq 0$, is endowed with an isomorphism $\gamma \circ (-\otimes \cO) \Rightarrow \gamma$ and is $2$-universal among all such functors. By construction, the category $\mathrm{D}/_{\!\!-\otimes \cO}$ is $R$-linear and additive. Moreover, it inherits naturally from $\mathrm{D}$ a (rigid) symmetric monoidal structure making the above functor \eqref{eq:functor} symmetric monoidal.
\section{Topological periodic cyclic homology}
In this section $F$ denotes a field of coefficients of characteristic zero. Moreover, we assume that $k$ is perfect. Thanks to the work of Hesselholt \cite[\S4]{Hesselholt}, topological periodic cyclic homology $TP$, which is defined as the Tate cohomology of the circle group action on topological Hochschild homology $THH$, yields a lax symmetric monoidal functor 
\begin{equation}\label{eq:TP}
TP_\pm(-)_{1/p} \colon \dgcat(k) \too \mathrm{Vect}_{\bbZ/2}(K)
\end{equation} 
with values in the category of $\bbZ/2$-graded $K$-vector spaces. Moreover, thanks to the work of Scholze (see \cite[Thm.~5.2]{positiveCD}), given a smooth proper $k$-scheme $X$, we have a natural isomorphism of (finite-dimensional) $\bbZ/2$-graded $K$-vector spaces 
\begin{equation}\label{eq:Scholze}
TP_\pm(\perf_\dg(X))_{1/p}\simeq (\bigoplus_{i\, \mathrm{even}}H^i_{\mathrm{crys}}(X), \bigoplus_{i\, \mathrm{odd}}H^i_{\mathrm{crys}}(X))\,,
\end{equation}
where $H^\ast_{\mathrm{crys}}(X):=H^\ast_{\mathrm{crys}}(X/W(k))\otimes_{W(k)}K$ stands for crystalline cohomology.
\begin{theorem}\label{thm:TP}
When $F\subseteq K$, resp. $K \subseteq F$, the lax symmetric monoidal functor \eqref{eq:TP} gives rise to a $F$-linear symmetric monoidal functor
\begin{equation}\label{eq:TP1}
TP_\pm(-)_{1/p}\colon \NChow(k)_F \too \mathrm{vect}_{\bbZ/2}(K)
\end{equation}
with values in the category of finite-dimensional $\bbZ/2$-graded $K$-vector spaces, resp. to a $F$-linear symmetric monoidal functor
\begin{equation}\label{eq:TP11}
TP_\pm(-)_{1/p}\otimes_K F\colon \NChow(k)_F \too \mathrm{vect}_{\bbZ/2}(F)
\end{equation}
with values in the category of finite-dimensional $\bbZ/2$-graded $F$-vector spaces.
\end{theorem}
\begin{proof}
Recall from \cite{BO1} that a {\em semi-orthogonal decomposition} of a triangulated category $\cT$, denoted by $\cT=\langle \cT_1, \cT_2\rangle$, consists of full triangulated subcategories $\cT_1, \cT_2 \subseteq \cT$ satisfying the following conditions: the inclusions $\cT_1, \cT_2 \subseteq \cT$ admit left and right adjoints; the triangulated category $\cT$ is generated by the objects of $\cT_1$ and $\cT_2$; and $\Hom_\cT(\cT_2, \cT_1)=0$. A functor $E\colon \dgcat(k)
\to \mathrm{D}$, with values in an additive category, is called an
{\em additive invariant} if it satisfies the following conditions:
\begin{itemize}
\item[(i)] It sends the Morita equivalences (see \S\ref{sub:dg}) to isomorphisms.
\item[(ii)] Given dg categories $\cA,\cC \subseteq \cB$ such that $\dgHo(\cB)=\langle\dgHo(\cA), \dgHo(\cC) \rangle$, the inclusions $\cA, \cC\subseteq \cB$ induce an isomorphism $E(\cA) \oplus E(\cC) \simeq E(\cB)$.
\end{itemize}
As explained in \cite[\S2.3]{book}, the functor $U(-)_R\colon \dgcat(k) \to \Hmo_0(k)_R$ is the {\em universal} additive invariant, \ie given any $R$-linear, additive, symmetric monoidal category $\mathrm{D}$, we have an induced equivalence of categories
\begin{equation}\label{eq:equivalence}
U(-)_R^\ast \colon \mathrm{Fun}^\otimes_{R\text{-}\mathrm{linear}}(\Hmo_0(k)_R, \mathrm{D}) \stackrel{\simeq}{\too} \mathrm{Fun}^\otimes_{\mathrm{add}}(\dgcat(k), \mathrm{D})\,,
\end{equation}
where the left-hand side denotes the category of $R$-linear lax symmetric monoidal functors and the right-hand side denotes the category of lax symmetric monoidal additive invariants. Since topological Hochschild homology $THH$ is a lax symmetric monoidal additive invariant (consult \cite[\S2.2.12]{book}), it follows from the exactness of the Tate construction that the above functor \eqref{eq:TP} is also a lax symmetric monoidal additive invariant. Consequently, making use of the field extension $F \subseteq K$, resp. $K \subseteq F$, and of the equivalence of categories \eqref{eq:equivalence}, we conclude that the functor \eqref{eq:TP}, resp. $TP_\pm(-)_{1/p}\otimes_k F$, gives rise to a $F$-linear lax symmetric monoidal functor 
\begin{equation}\label{eq:TP2}
TP_\pm(-)_{1/p} \colon \Hmo_0(k)_F \too \mathrm{Vect}_{\bbZ/2}(K)\,,
\end{equation}
resp. to a $F$-linear lax symmetric monoidal functor
\begin{equation}\label{eq:TP22}
TP_\pm(-)_{1/p} \otimes_K F \colon \Hmo_0(k)_F \too \mathrm{Vect}_{\bbZ/2}(F)\,.
\end{equation}
Given smooth proper dg categories $\cA$ and $\cB$, it is proved in \cite[Thm.~A]{BM} that the natural morphism of $TP(k)$-modules 
$$TP(\cA) \wedge_{TP(k)}TP(\cB) \too TP(\cA\otimes \cB)$$
is a weak equivalence. Consequently, the induced morphism of $TP(k)_{1/p}$-modules
\begin{equation}\label{eq:TP5}
TP(\cA)_{1/p} \wedge_{TP(k)_{1/p}}TP(\cB)_{1/p} \too TP(\cA\otimes \cB)_{1/p}
\end{equation}
is also a weak equivalence. Making use of the (convergent) Tor spectral sequence
$$ E^2_{i,j} = \mathrm{Tor}_{i,j}^{TP_\ast(k)_{1/p}}(TP_\ast(\cA)_{1/p}, TP_\ast(\cB)_{1/p}) \Rightarrow \mathrm{Tor}_{i+j}^{TP(k)_{1/p}} (TP(\cA)_{1/p}, TP(\cB)_{1/p})$$
and of the computation $TP_\ast(k)_{1/p} \simeq \mathrm{Sym}_K\{v^{\pm 1}\}$, where the divided Bott element $v$ is of degree $-2$ (see \cite[\S4]{Hesselholt}), we hence conclude that \eqref{eq:TP5} yields an isomorphism of $\bbZ/2$-graded $K$-vector spaces $TP_\pm(\cA)_{1/p} \otimes TP_\pm(\cB)_{1/p} \simeq TP_\pm(\cA\otimes \cB)_{1/p}$. By definition of the category of noncommutative Chow motives, this implies that \eqref{eq:TP2}, resp. \eqref{eq:TP22}, gives rise to a $F$-linear symmetric monoidal functor
\begin{equation}\label{eq:TP6}
TP_\pm(-)_{1/p}\colon \NChow(k)_F \too \mathrm{Vect}_{\bbZ/2}(K)\,,
\end{equation}
resp. to a $F$-linear symmetric monoidal functor
\begin{equation}\label{eq:TP66}
TP_\pm(-)_{1/p}\otimes_K F\colon \NChow(k)_F \too \mathrm{Vect}_{\bbZ/2}(F)\,.
\end{equation}
Finally, since the category $\NChow(k)_F$ is rigid and the dualizable objects of the category $\mathrm{Vect}_{\bbZ/2}(K)$, resp. $\mathrm{Vect}_{\bbZ/2}(F)$, are the finite-dimensional $\bbZ/2$-graded $K$-vector spaces, resp. $\bbZ/2$-graded $F$-vector spaces, the symmetric monoidal functor \eqref{eq:TP6}, resp. \eqref{eq:TP66}, takes values in $\mathrm{vect}_{\bbZ/2}(K)$, resp. $\mathrm{vect}_{\bbZ/2}(F)$.
\end{proof}
\subsection{Noncommutative homological motives}\label{sub:homological}
Assume that $F \subseteq K$, resp. $F \subseteq K$. Under these assumptions, the category of {\em noncommutative homological motives $\NHom(k)_F$} is defined as the idempotent completion of the quotient of the category $\NChow(k)_F$ by the kernel of the $F$-linear symmetric monoidal functor \eqref{eq:TP1}, resp. \eqref{eq:TP11}. By construction, $\NHom(k)_F$ is $F$-linear, additive, idempotent complete, and rigid symmetric monoidal. Moreover, we have canonical (quotient) functors:
$$ \NChow(k)_F \too \NHom(k)_F \too \NNum(k)_F\,.$$
\section{Semi-simplicity}\label{sec:semi}
In this section $F$ denotes a field of coefficients of characteristic zero. Recall from \S\ref{sub:numerical} and \cite[\S4]{Andre} the definition of the categories of (noncommutative) numerical motives $\NNum(k)_F$ and $\Num(k)_F$. Our first main result is the following:
\begin{theorem}\label{thm:main}
The category $\NNum(k)_F$ is abelian semi-simple.
\end{theorem}
Assuming certain (polarization) conjectures, Kontsevich conjectured in his seminal talk \cite{IAS} that the category of noncommutative numerical motives $\NNum(k)_F$ is abelian semi-simple. Theorem \ref{thm:main} not only proves this conjecture but moreover shows that Kontsevich's insight holds unconditionally.
\begin{corollary}\label{cor:main}
The category $\Num(k)_F$ is abelian semi-simple.
\end{corollary}
Assuming certain (standard) conjectures, Grothendieck conjectured in the sixties that the category $\Num(k)_F$ is abelian semi-simple. This conjecture was proved (unconditionally) by Jannsen \cite{Jannsen} in the nineties using \'etale cohomology. Corollary \ref{cor:main} provides us with an alternative proof of Grothendieck's conjecture.
\subsection*{Proof of Theorem \ref{thm:main}}
We assume first that $k$ is perfect. Let $F\subseteq F'$ be a field extension.  As proved in \cite[Lem.~4.29]{book}, if the category $\NNum(k)_F$ is abelian semi-simple, then the category $\NNum(k)_{F'}$ is also semi-simple. Hence, we can assume without loss of generality that $F=\bbQ$. Since $K$ is of characteristic zero, Theorem \ref{thm:TP} provides us with a $\bbQ$-linear symmetric monoidal functor 
\begin{equation}\label{eq:TP7}
TP_\pm(-)_{1/p} \colon \NChow(k)_\bbQ \too \mathrm{vect}_{\bbZ/2}(K)\,.
\end{equation}
Clearly, the $K$-linear category $\mathrm{vect}_{\bbZ/2}(K)$ is rigid symmetric monoidal and satisfies the conditions (i)-(ii) of Proposition \ref{prop:AK} below. Moreover, thanks to \eqref{eq:Homs}, we have an isomorphism $\mathrm{End}_{\NChow(k)_\bbQ}(U(k)_\bbQ)\simeq \bbQ$. Therefore, making use of Proposition \ref{prop:AK} below (with $H$ given by the symmetric monoidal functor \eqref{eq:TP7}), we conclude that the category $\NNum(k)_\bbQ$ is abelian semi-simple. 
\begin{proposition}{(\cite[Thm.~1]{AK1})}\label{prop:AK}
Let $(\mathrm{D},\otimes, {\bf 1})$ be a $F$-linear, additive, rigid symmetric monoidal category with $\mathrm{End}_{\mathrm{D}}({\bf 1})\simeq F$. Assume that there exists a symmetric monoidal functor $H\colon \mathrm{D} \to \mathrm{D}'$ with values in a $F'$-linear additive rigid symmetric monoidal category (with $F \subseteq F'$ a field extension) such that:
\begin{itemize}
\item[(i)] We have $\mathrm{dim}_{F'} \Hom_{\mathrm{D}'}(a,b)< \infty$ for any two objects $a$ and $b$.
\item[(ii)] Every nilpotent endomorphism in $\mathrm{D}'$ has a trivial categorical trace\footnote{This condition holds, for example, whenever $\mathrm{D}'$ is an abelian category.}.
\end{itemize}   
Under these assumptions, the idempotent completion of the quotient of $\mathrm{D}$ by the $\otimes$-ideal $\cN$ is an abelian semi-simple category.
\end{proposition}
Let us assume now that $k$ is arbitrary. As explained in \cite[\S4.10.2]{book}, the assignment $\cA \mapsto \cA\otimes_k k^{\mathrm{perf}}$, where $k^{\mathrm{perf}}$ stands for the perfect closure of $k$, gives rise to a $F$-linear symmetric monoidal functor:
\begin{equation}\label{eq:closure}
- \otimes_k k^{\mathrm{perf}}\colon \NChow(k)_F \too \NChow(k^{\mathrm{perf}})_F\,.
\end{equation}
Therefore, since $k^{\mathrm{perf}}$ is a perfect field, by applying Proposition \ref{prop:AK} to the following composition of $F$-linear symmetric monoidal functors
$$
\NChow(k)_F \stackrel{\eqref{eq:closure}}{\too} \NChow(k^{\mathrm{perf}})_F \too \NNum(k^{\mathrm{perf}})_F\,,
$$
we conclude from above that the category $\NNum(k)_F$ is abelian semi-simple. 
\subsection*{Proof of Corollary \ref{cor:main}}\label{sub:proof-Corollary}
Let $\Chow(k)_F$ be the classical category of Chow motives; consult \cite[\S4]{Andre}. As explained in \cite[\S4.2]{book}, there exists a $F$-linear, fully-faithful, symmetric monoidal functor $\Phi$ making the following diagram commute
\begin{equation}\label{eq:bridge}
\xymatrix@C=3em@R=1.7em{
\mathrm{SmProp}(k)^\op \ar[rr]^-{X \mapsto \perf_\dg(X)} \ar[d]_-{\mathfrak{h}(-)_F} && \dgcat_{\mathrm{sp}}(k) \ar[dd]^-{U(-)_F} \\
\Chow(k)_F \ar[d]_-{\gamma} & & \\
\Chow(k)_F/_{\!\!- \otimes F(1)} \ar[rr]_-\Phi && \NChow(k)_F \,,
}
\end{equation}
where $\mathrm{SmProp}(k)$ stands for the category of smooth proper $k$-schemes, $\dgcat_{\mathrm{sp}}(k)$ stands for the category of smooth proper dg categories, and $\Chow(k)_F/_{\!\!-\otimes F(1)}$ stands for the orbit category (see \S\ref{sub:orbit}) with respect to the Tate motive $F(1)$. Therefore, by applying Proposition \ref{prop:AK} to the $F$-linear symmetric monoidal functor
$$
\Chow(k)_F \stackrel{\gamma}{\too} \Chow(k)_F/_{\!\!- \otimes F(1)} \stackrel{\Phi}{\too} \NChow(k)_F \too \NNum(k)_F\,,
$$
we conclude from Theorem \ref{thm:main} that the category $\Num(k)_F$ (defined as the idempotent completion of the quotient $\Chow(k)_F/\cN$) is abelian semi-simple. 

\begin{remark}\label{rk:analogue}
The analogue of Theorem \ref{thm:main} (and of Corollary \ref{cor:main}), with $k$ of characteristic zero, was proved in \cite[Thm.~5.6]{JEMS} \cite[Thm.~1.10]{AJM} \cite[Thm.~1.2]{Comp}.
\end{remark}
\section{Zeta functions of endomorphisms}\label{app:1}
Let $N\!\!M \in \NChow(k)_\bbQ$ be a noncommutative Chow motive and $\mathrm{f}\colon N\!\!M \to N\!\!M$ an endomorphism. Following Kahn \cite[\S3]{Zeta}, the {\em zeta function of $\mathrm{f}$} is defined as the following formal power series
\begin{equation}\label{eq:zeta}
Z(\mathrm{f};t):= \mathrm{exp}(\sum_{n \geq 1} \mathrm{tr}(\mathrm{f}^{\circ n}) \frac{t^n}{n}) \in \bbQ\llbracket t \rrbracket\,,
\end{equation}
where $\mathrm{f}^{\circ n}$ stands for the composition of $\mathrm{f}$ with itself $n$-times, $\mathrm{tr}(\mathrm{f}^{\circ n}) \in \bbQ$ stands for the categorical trace of $\mathrm{f}^{\circ n}$, and $\mathrm{exp}(t):=\sum_{n \geq 0} \frac{t^n}{n!} \in \bbQ\llbracket t\rrbracket$. 
\begin{remark}[Hochschild homology with coefficients]
When $N\!\!M = U(\cA)_\bbQ$, with $\cA$ a smooth proper dg category, and $\mathrm{f}=[\mathrm{B}]_\bbQ$, with $\mathrm{B} \in \cD_c(\cA^\op \otimes \cA)$ a dg $\cA\text{-}\cA$-bimodule, we have the following computation
\begin{equation}\label{eq:integers}
\mathrm{tr}(\mathrm{f}^{\circ n}) = \sum_{i \geq 0} (-1)^i \mathrm{dim}\, HH_i(\cA; \underbrace{\mathrm{B}\otimes^{\bf L}_\cA \cdots \otimes^{\bf L}_\cA \mathrm{B}}_{n\text{-}\text{times}}) \in \bbZ\,,
\end{equation}
where $HH(\cA; \mathrm{B}\otimes^{\bf L}_\cA \cdots \otimes^{\bf L}_\cA \mathrm{B})$ stands for the Hochschild homology of $\cA$ with coefficients in the dg $\cA\text{-}\cA$-bimodule $\mathrm{B}\otimes^{\bf L}_\cA \cdots \otimes^{\bf L}_\cA \mathrm{B}$; see \cite[Prop.~2.26]{book}. Intuitively speaking, the integer \eqref{eq:integers} is the number of ``fixed points'' of $\mathrm{B}\otimes^{\bf L}_\cA \cdots \otimes^{\bf L}_\cA \mathrm{B}$.
\end{remark}
\begin{example}[Calabi-Yau dg categories]\label{ex:CY}
Thanks to the work of Bondal-Kapranov \cite{BK}, every smooth proper dg category $\cA$ comes equipped with a (unique) Serre dg $\cA\text{-}\cA$-bimodule $\mathrm{S}$. In the particular case where $\cA=\perf_\dg(X)$, with $X$ a smooth proper $k$-scheme, $\mathrm{S}$ is the dg $\cA\text{-}\cA$-bimodule associated to the Serre dg functor $-\otimes^{\bf L}_{\cO_X} K_X [\mathrm{dim}(X)]$, where $K_X$ stands for the canonical sheaf of $X$. Given an integer $d$, recall that a smooth proper dg category $\cA$ is called {\em $d$-Calabi-Yau} if $\mathrm{S}\simeq ({}_{\id} \mathrm{B})[d]$; consult \S\ref{sub:dg}. In the same vein, given integers $d$ and $r\neq 0$, $\cA$ is called {\em $\frac{d}{r}$-Calabi-Yau} if the $r$-fold tensor product $\mathrm{S} \otimes^{\bf L}_\cA \cdots \otimes^{\bf L}_\cA \mathrm{S}$ is isomorphic to $({}_{\id} \mathrm{B})[d]$.

Let $N\!\!M:=U(\cA)_\bbQ$ and $\mathrm{f}:=[\mathrm{S}]_\bbQ$. When $\cA$ is $d$-Calabi-Yau, we have $\mathrm{tr}(\mathrm{f}^{\circ n})= (-1)^{nd} \sum_{i \geq 0} (-1)^i\mathrm{dim}\, HH_i(\cA)$, where $HH(\cA)$ stands for the Hochschild homology of $\cA$. Similarly, when $\cA$ is $\frac{d}{r}$-Calabi-Yau, we have $(-1)^d \mathrm{tr}(\mathrm{f}^{\circ n})= \mathrm{tr}(\mathrm{f}^{\circ (n +r)})$.
\end{example}
\begin{example}[Classical zeta functions]\label{ex:classical}
Let $k=\bbF_q$ be a finite field and $X$~a~smooth proper $k$-scheme. Consider the noncommutative Chow motive $N\!\!M:=U(\perf_\dg(X))_\bbQ$ and the endomorphism $\mathrm{f}:=U(\mathrm{Fr}^\ast)_\bbQ$, where $\mathrm{Fr}$ stands for the Frobenius of $X$ and $\mathrm{Fr}^\ast\colon \perf_\dg(X) \to \perf_\dg(X)$ for the associated pull-back dg functor. In this particular case, we have $\mathrm{tr}(\mathrm{f}^{\circ n})= \langle {}^t \Gamma_{\mathrm{Fr}^{\circ n}}, \Delta \rangle = \#X(\bbF_{q^n})$, where ${}^t \Gamma_{\mathrm{Fr}^{\circ n}}$ is the transpose of the graph of $\mathrm{Fr}^{\circ n}$ and $\langle {}^t \Gamma_{\mathrm{Fr}^{\circ n}}, \Delta \rangle$ is the intersection number of ${}^t \Gamma_{\mathrm{Fr}^{\circ n}}$ with the diagonal $\Delta$ of $X \times X$. Consequently, the above formal power series \eqref{eq:zeta} reduces to the classical zeta function $Z_X(t):= \mathrm{exp}(\sum_{n \geq 1} \#X(\bbF_{q^n}) \frac{t^n}{n})$ of $X$.
\end{example}
\begin{example}[Classical non-abelian $L$-functions]\label{ex:L-functions}
Let $k=\bbF_q$ be a finite field, $G$ a finite group, $X$ a smooth proper $k$-scheme equipped with a $G$-action, and $\chi$ a complex character of $G$. Consider the noncommutative Chow motive $N\!\!M:=e U(\perf_\dg(X))_\bbC$, where $e$ stands for the idempotent $\frac{1}{|G|}\sum_{g \in G} \chi(g^{-1})U((g\cdot -)^\ast)_\bbC$, and the endomorphism $\mathrm{f}:=eU(\mathrm{Fr}^\ast)_\bbC$. In this particular case, we have
$$ \mathrm{tr}(\mathrm{f}^{\circ n})= \frac{1}{|G|} \sum_{g \in G} \chi(g^{-1}) \langle {}^t \Gamma_{\mathrm{Fr}^{\circ n}}, \Gamma_{(g \cdot -)}\rangle\,.$$
Consequently, the above formal power series \eqref{eq:zeta} reduces to the classical non-abelian $L$-function (consult Serre \cite[\S2.4]{Serre}):
$$L_{X, G, \chi}(t):= \mathrm{exp}(\sum_{n \geq 1}(\frac{1}{|G|} \sum_{g \in G} \chi(g^{-1}) \langle {}^t \Gamma_{\mathrm{Fr}^{\circ n}}, \Gamma_{(g \cdot -)}\rangle)\frac{t^n}{n}) \in \bbC\llbracket t\rrbracket\,.$$
Note that when $G$ is trivial, the $L$-function $L_{X,G,\chi}(t)$ reduces to the zeta function $Z_X(t)$. Moreover, when $\chi$ is trivial, the $L$-function $L_{X,G}(t)$ belongs to $\bbQ\llbracket t\rrbracket$.
\end{example}
\begin{example}[Orbifolds]\label{ex:orbifolds}
Let $k=\bbF_q$ be a finite field, $G$ a finite group such that $1/|G| \in \bbF_q$, and $X$ a smooth proper $k$-scheme equipped with a $G$-action. Consider the noncommutative Chow motive $N\!\!M:=U(\perf_\dg([X/G]))_\bbQ$, where $[X/G]$ stands for the (global) orbifold associated to the $G$-action, and the endomorphism $\mathrm{f}:=U(\mathrm{Fr}^\ast)_\bbQ$. In this particular case, thanks to \cite[Cor.~1.5(ii)]{Orbifold} and to the fact that the base-change functor $\Chow(k)_\bbQ \to \Chow(\overline{k})_\bbQ$ and the change-of-coefficients functor $\Chow(k)_\bbQ \to \Chow(k)_\bbC$ are symmetric monoidal, we have 
$$ \mathrm{tr}(\mathrm{f}^{\circ n}) = \sum_{g \in G\!/\!\sim} (\frac{1}{|C(g)|} \sum_{h \in C(g)} \langle {}^t \Gamma_{\mathrm{Fr}^{\circ n}_{|X^g}}, \Gamma_{(h\cdot - )_{|X^g}} \rangle)\,,$$
where $G\!/\!\!\!\sim$ stands for a (chosen) set of representatives of the conjugacy classes of $G$, $C(g)$ stands for the centralizer of $g$, and $(h\cdot -)_{|X^g}$ and $\mathrm{Fr}^{\circ n}_{|X^g}$ stand for the restrictions of the endomorphisms $(h\cdot -)$ and $\mathrm{Fr}^{\circ n}$ of $X$ to the closed subscheme $X^g$. Consequently, the above formal power series \eqref{eq:zeta} reduces to the following product of classical non-abelian $L$-functions $\prod_{g \in G\!/\!\sim}L_{X^g, C(g)}(t)$.
\end{example}
\begin{remark}[Witt vectors]\label{rk:Witt}
Recall from \cite{Hazewinkel} the definition of the ring of (big) Witt vectors $\mathrm{W}(\bbQ)=(1 + t \bbQ\llbracket t \rrbracket, \times, \ast)$. Since the leading term of \eqref{eq:zeta} is equal to $1$, the zeta function $Z(\mathrm{f};t)$ of $\mathrm{f}$ belongs to $\mathrm{W}(\bbQ)$. Moreover, given endomorphisms $\mathrm{f}$ and $\mathrm{f}'$ of noncommutative Chow motives $N\!\!M$ and $N\!\!M'$, respectively, we have $Z(\mathrm{f}\oplus \mathrm{f}';t)=Z(\mathrm{f};t) \times Z(\mathrm{f}';t)$ and $Z(\mathrm{f}\otimes \mathrm{f}'; t) = Z(\mathrm{f};t) \ast Z(\mathrm{f}';t)$ in $\mathrm{W}(\bbQ)$.
\end{remark}
\subsection*{Rationality and functional equation}
Let $B = \prod_i B_i$ be a finite-dimensional semi-simple $\bbQ$-algebra. Following Kahn \cite[\S1]{Zeta}, let us write $Z_i$ for the center of $B_i$, $\delta_i$ for the degree $[Z_i:\bbQ]$, and $d_i$ for the index $[B_i: Z_i]^{1/2}$. Given a unit $b \in B^\times$, its {\em $i^{\mathrm{th}}$ reduced norm} $\mathrm{Nrd}_i(b) \in \bbQ$ is defined as $(\mathrm{N}_{Z_i/\bbQ}\circ \mathrm{Nrd}_{B_i/Z_i})(b_i)$. 
\begin{definition}[Determinant]\label{def:determinant} 
Let $N\!\!M \in \NChow(k)_\bbQ$ be a noncommutative Chow motive. Thanks to Theorem \ref{thm:main}, $B:=\mathrm{End}_{\NNum(k)_\bbQ}(N\!\!M)$ is a finite-dimensional semi-simple $\bbQ$-algebra; let us write $e_i \in B$ for the central idempotent corresponding to the summand $B_i$. Given an invertible endomorphism $\mathrm{f}\colon N\!\!M \to N\!\!M$, its {\em determinant $\mathrm{det}(\mathrm{f}) \in \bbQ$} is defined as $\prod_i \mathrm{Nrd}_i(\mathrm{f})^{\mu_i}$, where $\mu_i :=\frac{\mathrm{tr}(e_i)}{\delta_i d_i}\in \bbZ$.
\end{definition}
Note that the above definition of determinant is ``intrinsic'' to the theory of noncommutative motives. Our second main result is the following:
\begin{theorem}\label{thm:zeta}
\begin{itemize}
\item[(i)] The formal power series $Z(\mathrm{f};t) \in \bbQ \llbracket t \rrbracket$ is {\em rational}, \ie it belongs to the subfield $\bbQ(t)\subset \bbQ\llbracket t\rrbracket$. Moreover, we have $\mathrm{deg}(Z(\mathrm{f};t))=-\mathrm{tr}(\id_{N\!\!M})$. 
\item[(ii)] When $\mathrm{f}$ is invertible, we have the following functional equation in $\bbQ(t)$
\begin{equation}\label{eq:functional}
Z((\mathrm{f}^{-1})^\vee;t^{-1})  = (-t)^{\mathrm{tr}(\id_{N\!\!M})}\cdot  \mathrm{det}(\mathrm{f}) \cdot Z(\mathrm{f};t)\,,
\end{equation}
where $(-)^\vee$ stands for the duality functor of $\NChow(k)_\bbQ$; consult \S\ref{sub:Chow}.
\item[(iii)]
We have the following equality
\begin{equation}\label{eq:equality-det}
\mathrm{det}(\mathrm{f})=\frac{\mathrm{det}(TP_+(\mathrm{f}\otimes_kk^{\mathrm{perf}})_{1/p}\,|\, TP_+(N\!\!M\otimes_kk^{\mathrm{perf}})_{1/p})}{\mathrm{det}(TP_-(\mathrm{f}\otimes_kk^{\mathrm{perf}})_{1/p}\,|\, TP_-(N\!\!M\otimes_kk^{\mathrm{perf}})_{1/p})}\,,
\end{equation}
where $k^{\mathrm{perf}}$ stands for the perfect closure of $k$.
\end{itemize}
\end{theorem}
Note that item (iii) provides a cohomological interpretation of the determinant.
\begin{remark}[Characteristic zero]\label{rk:zero}
As pointed out by Kontsevich \cite{Bonn}, Theorem \ref{thm:zeta} holds similarly when $k$ is of characteristic zero: simply replace Theorem \ref{thm:main}, resp. the functor \eqref{eq:TP1}, by Remark \ref{rk:analogue}, resp. by the periodic cyclic homology functor $HP_\pm\colon \NChow(k)_\bbQ \to \mathrm{vect}_{\bbZ/2}(k)$ (consult \cite[\S7]{JEMS}), and run the same proof.
\end{remark}
\begin{remark}[Artin-Mazur zeta functions]
Let $S$ be a set and $f\colon S \to S$ an endomorphism. Following \cite{AM}, the {\em Artin-Mazur zeta function of $f$} is defined as the formal power series $Z_{\mathrm{AM}}(f;t):=\mathrm{exp}(\sum_{n \geq 1} \# \mathrm{Fix}(f^{\circ n})\frac{t^n}{n})\in \bbQ\llbracket t \rrbracket$, where $\# \mathrm{Fix}(f^{\circ n})$ stands for the number of fixed points of $f^{\circ n}$ (which we assume to be finite for every $n \geq 1$). In contrast with Theorem \ref{thm:zeta}(i), the Artin-Mazur zeta function is {\em not} always rational. For example, take for $S$ the set $\bbP^1(\overline{\bbF_p})$, with $p$ a prime number, and for $f$ the power map $x \mapsto x^m$, with $m$ an integer coprime to $p$. In this case, as proved by Bridy in \cite[Thms.~1.2 and 1.3]{Bridy1}\cite[Thm.~1]{Bridy2}, the associated Artin-Mazur zeta function is {\em not} rational. Roughly speaking, the reason for this phenomenon is that the integer $\# \mathrm{Fix}(f^{\circ n})$ does {\em not} keeps track of the multiplicity of the fixed points of $f^{\circ n}$ (in contrast, the integer $\mathrm{tr}(\mathrm{f}^{\circ n})$ used in the definition of the formal power series \eqref{eq:zeta} keeps track of the multiplicity of the fixed points of $\mathrm{f}^{\circ n}$).
\end{remark}
\begin{proof}
We start by proving item (i). Thanks to the base-change functor \eqref{eq:closure} (with $F=\bbQ$), the formal power series $Z(\mathrm{f};t)$ agrees with the formal power series $Z(\mathrm{f}\otimes_k k^{\mathrm{perf}};t)$. Therefore, we can assume without loss of generality that $k$ is perfect. Thanks to Theorem \ref{thm:TP}, we have the equalities of formal power series:
\begin{eqnarray}
Z(\mathrm{f};t) & = & \mathrm{exp}(\sum_{n\geq 1}\mathrm{tr}(TP_\pm(\mathrm{f}^{\circ n})_{1/p})\frac{t^n}{n})\nonumber\\
& = &\mathrm{exp}(\sum_{n \geq 1} (\mathrm{tr}(TP_+(\mathrm{f}^{\circ n})_{1/p})-\mathrm{tr}(TP_-(\mathrm{f}^{\circ n})_{1/p}))\frac{t^n}{n}) \nonumber \\
& = & \mathrm{exp}(\sum_{n \geq 1} \mathrm{tr}(TP_+(\mathrm{f}^{\circ n})_{1/p})\frac{t^n}{n}-\sum_{n \geq 1}\mathrm{tr}(TP_-(\mathrm{f}^{\circ n})_{1/p})\frac{t^n}{n}) \nonumber \\
& = & \frac{\mathrm{exp}(\sum_{n \geq 1} \mathrm{tr}(TP_+(\mathrm{f}^{\circ n})_{1/p})\frac{t^n}{n})}{\mathrm{exp}(\sum_{n \geq 1} \mathrm{tr}(TP_-(\mathrm{f}^{\circ n})_{1/p})\frac{t^n}{n})} \in K\llbracket t \rrbracket \label{eq:quocient} \,. 
\end{eqnarray}
Making use of \cite[Lem.~27.5]{Milne}, we observe that \eqref{eq:quocient} agrees with
$$  \frac{\mathrm{exp}(\mathrm{log}(\mathrm{det}(\id - t \,TP_+(\mathrm{f})_{1/p}\,|\,TP_+(N\!\!M)_{1/p})^{-1}))}{\mathrm{exp}(\mathrm{log}(\mathrm{det}(\id - t \,TP_-(\mathrm{f})_{1/p}\,|\,TP_-(N\!\!M)_{1/p})^{-1}))} \,.$$
Consequently, we obtain the following equality of formal power series:
\begin{equation}\label{eq:equation-main}
 Z(\mathrm{f};t) = \frac{\mathrm{det}(\id - t \,TP_-(\mathrm{f})_{1/p}\,|\, TP_-(N\!\!M)_{1/p})}{\mathrm{det}(\id - t\, TP_+(\mathrm{f})_{1/p}\,|\, TP_+(N\!\!M)_{1/p})} \in K\llbracket t \rrbracket\,.
\end{equation}
This shows that $Z(\mathrm{f};t)$ is rational as a formal power series with $K$-coefficients. Thanks to \cite[Lem.~27.9]{Milne}, $Z(\mathrm{f};t)$ is also rational as a formal power series with $\bbQ$-coefficients, \ie we have an equality $Z(\mathrm{f};t)=\frac{p(t)}{q(t)}$ with $p(t), q(t) \in \bbQ[t]$. Moreover, we have the following equalities:
\begin{eqnarray*}
\mathrm{deg}(Z(\mathrm{f};t)) & = & \mathrm{deg}(p(t)) - \mathrm{deg}(q(t)) \\
& = & \mathrm{deg}(\mathrm{det}(\id - t \,TP_-(\mathrm{f})_{1/p})) - \mathrm{deg}(\mathrm{det}(\id - t \,TP_+(\mathrm{f})_{1/p})) \\
& = & \mathrm{dim}\,TP_-(N\!\!M)_{1/p} - \mathrm{dim}\,TP_+(N\!\!M)_{1/p} \\
& = & -\mathrm{tr}(\id_{N\!\!M})\,.
\end{eqnarray*}
We now prove item (ii). The formal power series $Z(\mathrm{f};t)$ and $Z((\mathrm{f}^{-1})^\vee;t^{-1})$, as well as the categorical trace $\mathrm{tr}(\id_{N\!\!M})$, can be computed in $\NChow(k)_\bbQ$ or in the abelian semi-simple category $\NNum(k)_\bbQ$. Therefore, the functional equation \eqref{eq:functional} follows from \cite[Prop.~2.2 and Thm.~3.2]{Zeta}.

We now prove item (iii). Let us assume first that $k$ is perfect. Under this assumption, note that \eqref{eq:equation-main} (with $\mathrm{f}$ and $t$ replaced by $(\mathrm{f}^{-1})^\vee$ and $t^{-1}$, respectively) yields the following equality:
\begin{equation}\label{eq:equality1}
 Z((\mathrm{f}^{-1})^\vee;t^{-1}) = \frac{\mathrm{det}(\id - t^{-1} \,TP_-((\mathrm{f}^{-1})^\vee)_{1/p}\,|\, TP_-(N\!\!M^\vee)_{1/p})}{\mathrm{det}(\id - t^{-1}\, TP_+((\mathrm{f}^{-1})^\vee)_{1/p}\,|\, TP_+(N\!\!M^\vee)_{1/p})}\,.
\end{equation}
Since the object $N\!\!M \in \NChow(k)_\bbQ$ is dualizable, it is well-known that the evaluation pairing $\mathrm{ev}\colon N\!\!M \otimes N\!\!M^\vee \to U(k)_\bbQ$ is invariant under the endomorphism $\mathrm{f} \otimes (\mathrm{f}^{-1})^\vee$; consult \cite[\S3.2.3.6]{Rivano}. Consequently, by applying the symmetric monoidal functor \eqref{eq:TP1} (with $F=\bbQ$), we obtain the following commutative diagram of finite-dimensional $\bbZ/2$-graded $K$-vector spaces:
\begin{equation*}\label{eq:square2}
\xymatrix{
TP_\pm(N\!\!M)_{1/p} \otimes TP_\pm(N\!\!M)^\vee_{1/p} \ar[d]_-{TP_\pm(\mathrm{f})_{1/p} \otimes TP_\pm((\mathrm{f}^{-1})^\vee)_{1/p}} \ar[rrr]^-{TP_\pm(\mathrm{ev})_{1/p}} &&& TP_\pm(k)_{1/p} \ar@{=}[d] \\
TP_\pm(N\!\!M)_{1/p} \otimes TP_\pm(N\!\!M)^\vee_{1/p} \ar[rrr]_-{TP_\pm(\mathrm{ev})_{1/p}} &&& TP_\pm(k)_{1/p}\,.
}
\end{equation*}
The pairing $TP_\pm(\mathrm{ev})_{1/p}$ is perfect. Therefore, making use of \cite[App.~C, Lem.~4.3]{Hartshorne}, we observe that the numerator of the right-hand side of \eqref{eq:equality1}~is~equal~to 
$$\frac{(-1)^{\mathrm{dim}\, TP_-(N\!\!M)_{1/p}} \cdot t^{- \mathrm{dim}\, TP_-(N\!\!M)_{1/p}}}{\mathrm{det}(TP_-(\mathrm{f})_{1/p}\,|\, TP_-(N\!\!M)_{1/p})} \cdot \mathrm{det}(\id - t \,TP_-(\mathrm{f})_{1/p}\,|\,TP_-(N\!\!M)_{1/p})$$
and, similarly, that the denominator of the right-hand side of \eqref{eq:equality1} is equal to 
$$\frac{(-1)^{\mathrm{dim}\, TP_+(N\!\!M)_{1/p}} \cdot t^{- \mathrm{dim}\,TP_+(N\!\!M)_{1/p}}}{\mathrm{det}(TP_+(\mathrm{f})_{1/p}\,|\, TP_+(N\!\!M)_{1/p})} \cdot \mathrm{det}(\id - t \,TP_+(\mathrm{f})_{1/p}\,|\,TP_+(N\!\!M)_{1/p})\,.$$
Thanks to the equalities $\mathrm{dim}\,TP_+(N\!\!M)_{1/p} -  \mathrm{dim}\, TP_-(N\!\!M)_{1/p} = \mathrm{tr}(\id_{N\!\!M})$ and \eqref{eq:equation-main}, this yields the following functional equation in $K(t)$:
\begin{equation}\label{eq:functional2}
Z((\mathrm{f}^{-1})^\vee;t^{-1})  = (-t)^{\mathrm{tr}(\id_{N\!\!M})} \cdot \frac{\mathrm{det}(TP_+(\mathrm{f})_{1/p}\,|\,TP_+(N\!\!M)_{1/p})}{\mathrm{det}(TP_-(\mathrm{f})_{1/p}\,|\,TP_-(N\!\!M)_{1/p})} \cdot Z(\mathrm{f};t) \,.
\end{equation}
By comparing the functional equations \eqref{eq:functional} and \eqref{eq:functional2}, we hence obtain the searched equality \eqref{eq:equality-det} (in the case where $k$ is perfect). Finally, let us assume that $k$ is arbitrary. Thanks to the base-change functor \eqref{eq:closure} (with $F=\bbQ$), the formal power series $Z(\mathrm{f};t)$ and $Z((\mathrm{f}^{-1})^\vee;t)$ agree with $Z(\mathrm{f}\otimes_k k^{\mathrm{perf}};t)$ and $Z(((\mathrm{f}\otimes_k k^{\mathrm{perf}})^{-1})^\vee;t)$, respectively. Moreover, the categorical trace $\mathrm{tr}(\id_{N\!\!M})$ agrees with $\mathrm{tr}(\id_{N\!\!M\otimes_k k^{\mathrm{perf}}})$. Therefore, by comparing the functional equation \eqref{eq:functional} with the functional equation obtained from \eqref{eq:functional2} by replacing $N\!\!M$ and $\mathrm{f}$ by $N\!\!M\otimes_k k^{\mathrm{perf}}$ and $\mathrm{f}\otimes_k k^{\mathrm{perf}}$, respectively, we obtain the searched equality \eqref{eq:equality-det}.
\end{proof}
\begin{remark}
Note that in contrast with the proof of the functional equation \eqref{eq:functional}, the proof of the functional equation \eqref{eq:functional2} does {\em not} makes use of Kahn's work \cite{Zeta}.
\end{remark}
\begin{corollary}[Weil conjectures]\label{cor:zeta}
Let $k=\bbF_q$ be a finite field, $X$ a smooth proper $k$-scheme $X$, and $\mathrm{E}:=\langle \Delta, \Delta \rangle$ the self-intersection number of the diagonal of $X \times X$.
\begin{itemize}
\item[(i)] The zeta function $Z_X(t)$ of $X$ is rational. Moreover, $\mathrm{deg}(Z_X(t))=-\mathrm{E}$. 
\item[(ii)] We have the functional equation $Z_X(\frac{1}{q^{\mathrm{dim}(X)} t}) = \pm t^{\mathrm{E}} q^{\frac{\mathrm{dim}(X)}{2} \mathrm{E}} Z_X(t)$.
\end{itemize}
\end{corollary}
\begin{proof}
Let $N\!\!M$ and $\mathrm{f}$ be as in Example \ref{ex:classical}. Since $Z(\mathrm{f};t)=Z_X(t)$ and $\mathrm{tr}(\id_{N\!\!M})=\mathrm{E}$, the proof of item (i) follows from Theorem \ref{thm:zeta}(i). In what concerns item (ii), note first that the natural isomorphism \eqref{eq:Scholze} yields the following equality:
\begin{equation}\label{eq:equality-dett}
\frac{\mathrm{det}(TP_+(\mathrm{f})_{1/p}\,|\,TP_+(N\!\!M)_{1/p})}{\mathrm{det}(TP_-(\mathrm{f})_{1/p}\,|\,TP_-(N\!\!M)_{1/p})}= \frac{\prod_{i \, \mathrm{even}}\mathrm{det}(H^i_{\mathrm{crys}}(\mathrm{Fr})\,|\, H^i_{\mathrm{crys}}(X))}{\prod_{i \, \mathrm{odd}}\mathrm{det}(H^i_{\mathrm{crys}}(\mathrm{Fr})\,|\, H^i_{\mathrm{crys}}(X))}\,.
\end{equation}
Making use of \cite[App.~C, Lem.~4.3]{Hartshorne}, we observe that the square of the right-hand side of \eqref{eq:equality-dett} is equal to $q^{\mathrm{dim}(X)\mathrm{E}}$. Consequently, $\eqref{eq:equality-dett}=\pm q^{\frac{\mathrm{dim}(X)}{2} \mathrm{E}}$. Since $Z((\mathrm{f}^{-1})^\vee; t^{-1})=Z_X(\frac{1}{q^{\mathrm{dim}(X)}t})$, item (ii) follows then from Theorem \ref{thm:zeta}(ii).
\end{proof}
Weil \cite{Weil} conjectured\footnote{Weil conjectured also in {\em loc. cit.} that the zeta function $Z_X(t)$ of $X$ satisfies the analogue of the Riemann hypothesis. This conjecture was proved by Deligne \cite{Deligne} in the seventies.} in the late forties that the zeta function $Z_X(t)$ of $X$ is rational and that it satisfies the functional equation $Z_X(\frac{1}{q^{\mathrm{dim}(X)} t}) = \pm t^{\mathrm{E}} q^{\frac{\mathrm{dim}(X)}{2} \mathrm{E}} Z_X(t)$. These conjectures were proved independently by Dwork \cite{Dwork} and Grothendieck \cite{Grothendieck} using $p$-adic analysis and \'etale cohomology, respectively. The above Corollary \ref{cor:zeta} provides us with an alternative proof of the Weil conjectures. Moreover, Theorem \ref{thm:zeta} (and Remark \ref{rk:zero}) establishes a far-reaching noncommutative generalization of Dwork's and Grothendieck's results. 
\begin{remark}[Related work]
Let $k=\bbF_q$ be a finite field. In a recent work \cite{NCWeil}, we developed a general theory of zeta functions of noncommutative Chow motives (without endomorphisms). As explained in {\em loc. cit.}, given a noncommutative Chow motive $N\!\!M \in \NChow(k)_\bbQ$, the finite-dimensional $K$-vector space $TP_0(N\!\!M)_{1/p}$, resp. $TP_1(N\!\!M)_{1/p}$, comes equipped with an automorphism $\mathrm{F}_0$, resp. $\mathrm{F}_1$, called the ``cyclotomic Frobenius''. Under these notations, the {\em even/odd zeta function of $N\!\!M$} is defined as the following formal power series:
\begin{eqnarray}
Z_{\mathrm{even}}(N\!\!M;t) & := & \mathrm{det}(\id - t\,\mathrm{F}_0\,|\,TP_0(N\!\!M)_{1/p})^{-1} \in K \llbracket t \rrbracket \label{eq:even}\\
Z_{\mathrm{odd}}(N\!\!M;t) & := & \mathrm{det}(\id - t\,\mathrm{F}_1\,|\,TP_1(N\!\!M)_{1/p})^{-1} \in K \llbracket t \rrbracket\,. \label{eq:odd}
\end{eqnarray}
As explained in {\em loc. cit.}, the cyclotomic Frobenius is {\em not} compatible with the $\bbZ/2$-graded structure of $TP_\ast(N\!\!M)_{1/p}$. Instead, we have canonical isomorphisms $\mathrm{F}_n\simeq q\cdot \mathrm{F}_{n+2}$ for every $n \in \bbZ$. For this reason, in the case of noncommutative Chow motives without endomorphisms, it is more natural to consider separately the even zeta function \eqref{eq:even} and the odd zeta function \eqref{eq:odd} (instead of a single zeta function \eqref{eq:equation-main} as in the case of noncommutative Chow motives with endomorphisms). Finally, as explained in \cite[Thm.~1.12]{NCWeil}, the even and odd zeta functions also satisfy appropriate functional equations.
\end{remark}
\subsection*{Exponential growth rate}
Let $N\!\!M \in \Chow(k)_\bbQ$ be a noncommutative Chow motive and $\mathrm{f}\colon N\!\!M \to N\!\!M$ an endomorphism. Consider the associated sequence of rational numbers $\{|\mathrm{tr}(\mathrm{f}^{\circ n})|\}_{n \geq 1}$. Its {\em exponential growth rate} is defined as follows: 
$$ \mathrm{rate}(\mathrm{f}):=\mathrm{limsup}_{n \to \infty} \frac{1}{n} \mathrm{log}(|\mathrm{tr}(\mathrm{f}^{\circ n})|) \in [-\infty, +\infty]\,.$$
\begin{example}[Calabi-Yau dg categories]
Let $N\!\!M$ and $\mathrm{f}$ (and $\cA$) be as in Example \ref{ex:CY}. In the case where $\cA$ is $d$-Calabi-Yau, the associated sequence $\{|\mathrm{tr}(\mathrm{f}^{\circ n})|\}_{n \geq 1}$ is constant. Similarly, in the case where $\cA$ is $\frac{d}{r}$-Calabi-Yau, the associated sequence $\{|\mathrm{tr}(\mathrm{f}^{\circ n})|\}_{n \geq 1}$ is periodic. Consequently, in both these cases, we have $\mathrm{rate}(\mathrm{f})=0$.
\end{example}
\begin{notation}\label{not:endomorphisms}
Given an embedding $\iota\colon K \hookrightarrow \bbC$, consider the induced endomorphism of finite-dimensional $\bbZ/2$-graded $\bbC$-vector spaces:
$$ \theta_\pm(\mathrm{f}; \iota):= TP_\pm(\mathrm{f})_{1/p} \otimes_{K, \iota} \bbC \colon TP_\pm(N\!\!M)_{1/p}\otimes_{K, \iota} \bbC \too TP_\pm(N\!\!M)_{1/p}\otimes_{K, \iota} \bbC\,.$$
Let $\rho_+(\mathrm{f}; \iota)$, resp. $\rho_-(\mathrm{f};\iota)$, be the spectral radius of $\theta_+(\mathrm{f}; \iota)$, resp. $\theta_-(\mathrm{f}; \iota)$, and $\rho(\mathrm{f};\iota):=\mathrm{max}\{\rho_+(\mathrm{f};\iota), \rho_-(\mathrm{f};\iota)\}$. 
\end{notation}
\begin{lemma}\label{lem:bound}
Given an embedding $\iota\colon K \hookrightarrow \bbC$, we have the following inequalities:
\begin{eqnarray*}
|\mathrm{tr}(\mathrm{f}^{\circ n})| \leq (\mathrm{dim}\,TP_+(N\!\!M)_{1/p} + \mathrm{dim}\,TP_-(N\!\!M)_{1/p})\cdot \rho(\mathrm{f};\iota)^n && n \geq 1\,.
\end{eqnarray*}
\end{lemma}
\begin{proof}
In order to simplify the exposition, let us write $d_+$, resp. $d_-$, for the dimension of the $K$-vector space $TP_+(N\!\!M)_{1/p}$, resp. $TP_-(N\!\!M)_{1/p}$, and $d:=d_++d_-$. Under these notations, the proof follows then from the following (in)equalities
\begin{eqnarray}
|\mathrm{tr}(\mathrm{f}^{\circ n})| & = & |\mathrm{tr}(\theta_\pm(\mathrm{f}^{\circ n}; \iota))| \label{eq:star1}\\
& = & |\mathrm{tr}(\theta_+(\mathrm{f}^{\circ n}; \iota)) - \mathrm{tr}(\theta_-(\mathrm{f}^{\circ n}; \iota))| \nonumber \\
& \leq & |\mathrm{tr}(\theta_+(\mathrm{f}^{\circ n}; \iota))| + |\mathrm{tr}(\theta_-(\mathrm{f}^{\circ n}; \iota))| \nonumber \\
& = & |\lambda^n_{+, 1}+ \cdots + \lambda^n_{+, d_+}| + |\lambda^n_{-, 1}+ \cdots + \lambda^n_{-, d_-}| \nonumber \\
& \leq & |\lambda_{+, 1}|^n + \cdots + |\lambda_{+, d_+}|^n + |\lambda_{-, 1}|^n+ \cdots + |\lambda_{-, d_-}|^n \nonumber \\
& \leq & d_+ \cdot \rho_+(\mathrm{f};\iota)^n + d_- \cdot \rho_-(\mathrm{f};\iota)^n \nonumber \\
& \leq & d\cdot \rho(\mathrm{f};\iota)^n\,, \nonumber
\end{eqnarray}
where the equality \eqref{eq:star1} is a consequence of Theorem \ref{thm:TP} and $\{\lambda_{+,1}, \ldots, \lambda_{+, d_+}\}$, resp. $\{\lambda_{-,1}, \ldots, \lambda_{-, d_-}\}$, stands for the eigenvalues (with multiplicity) of the endomorphism $\theta_+(\mathrm{f};\iota)$, resp. $\theta_-(\mathrm{f};\iota)$.
\end{proof}
Note that Lemma \ref{lem:bound} yields automatically the following upper bound:
\begin{corollary}
Given an embedding $\iota\colon K \hookrightarrow \bbC$, we have $ \mathrm{rate}(\mathrm{f}) \leq \mathrm{log}(\rho(\mathrm{f};\iota))$.
\end{corollary}
\begin{proposition}\label{prop:radius}
Assume that the endomorphisms $\theta_+(\mathrm{f};\iota)$ and $\theta_-(\mathrm{f};\iota)$ have different eigenvalues, with multiplicities, on the circle $\{\lambda \in \bbC\,|\, |\lambda| = \rho(\mathrm{f}; \iota)\}$. Under this assumption, the equality $\mathrm{rate}(\mathrm{f})=\mathrm{log}(\rho(\mathrm{f}; \iota))$ holds.
\end{proposition}
\begin{proof}
Following \cite[Lem.~2.8]{KK}, we have the equality:
$$\rho(\mathrm{f};\iota) = \mathrm{limsup}_{n \to \infty} |\mathrm{tr}(\theta_+(\mathrm{f}^{\circ n}; \iota)) - \mathrm{tr}(\theta_-(\mathrm{f}^{\circ n};\iota))|^{\frac{1}{n}}\,.$$
Therefore, since $\mathrm{tr}(\mathrm{f}^{\circ n})= \mathrm{tr}(\theta_+(\mathrm{f}^{\circ n};\iota))-\mathrm{tr}(\theta_-(\mathrm{f}^{\circ n};\iota))$, we conclude that
\begin{equation}\label{eq:3-equalities}
\rho(\mathrm{f};\iota)= \mathrm{limsup}_{n \to \infty} |\mathrm{tr}(\mathrm{f}^{\circ n})|^{\frac{1}{n}} = \mathrm{exp}(\mathrm{limsup}_{n \to \infty} \frac{1}{n} \mathrm{log} (|\mathrm{tr}(\mathrm{f}^{\circ n})|))\,.
\end{equation}
By applying $\mathrm{log}(-)$ to \eqref{eq:3-equalities}, we hence obtain the searched equality.
\end{proof}
\begin{example}[Classical zeta functions]\label{ex:classical2}
Let $N\!\!M$ and $\mathrm{f}$ (and $X$) be as in Example \ref{ex:classical}. Thanks to the natural isomorphism \eqref{eq:Scholze} and to Deligne's proof \cite{Deligne} of the analogue of the Riemann hypothesis (consult also Katz-Messing \cite{KM}), we have $\rho_+(\mathrm{f};\iota)\neq \rho_-(\mathrm{f};\iota)$ and $\rho(\mathrm{f};\iota)= q^{\mathrm{dim}(X)}$ (independently of the embedding $\iota$). Therefore, making use of Proposition \ref{prop:radius}, we conclude that $\mathrm{rate}(\mathrm{f})= \mathrm{log}(q^{\mathrm{dim}(X)})$.
\end{example}
\begin{example}[Classical non-abelian $L$-functions]\label{ex:L-functions2}
Let $N\!\!M$ and $\mathrm{f}$ (and $G$, $X$ and $\chi$) be as in Example \ref{ex:L-functions}. Similarly to Example \ref{ex:classical2}, we have $\rho_+(\mathrm{f};\iota)\neq \rho_-(\mathrm{f};\iota)$ and $\rho(\mathrm{f};\iota)\leq  q^{\mathrm{dim}(X)}$ (independently of the embedding $\iota$). Therefore, making use of Proposition \ref{prop:radius}, we conclude that $\mathrm{rate}(\mathrm{f})\leq\mathrm{log}(q^{\mathrm{dim}(X)})$.
\end{example}
\begin{example}[Orbifolds]
Let $N\!\!M$ and $\mathrm{f}$ (and $G$ and $X$) be as in Example \ref{ex:orbifolds}. Similarly to Example \ref{ex:L-functions2}, $\rho_+(\mathrm{f};\iota)\neq \rho_-(\mathrm{f};\iota)$ and $\rho(\mathrm{f};\iota)\leq \mathrm{max}\{q^{\mathrm{dim}(X^g)}\,|\,g \in G\!/\!\!\!\sim\}$ (independently of the embedding $\iota$). Therefore, making use of Proposition \ref{prop:radius}, we conclude that $\mathrm{rate}(\mathrm{f})\leq\mathrm{max}\{\mathrm{log}(q^{\mathrm{dim}(X^g)})\,|\,g \in G\!/\!\!\!\sim\}$.
\end{example}
\section{Hasse-Weil zeta functions of endomorphisms}\label{sec:app-2}
In this section we assume that $k=\bbF_q$ is a finite field. Let $N\!\!M \in \NChow(k)_\bbQ$ be a noncommutative Chow motive and $\mathrm{f}\colon N\!\!M \to N\!\!M$ an endomorphism. The {\em Hasse-Weil zeta function of $\mathrm{f}$} is defined as the following Dirichlet series:
\begin{equation}\label{eq:function}
\zeta(\mathrm{f};s) := Z(\mathrm{f};q^{-s})\,.
\end{equation}
\begin{proposition}\label{prop:convergence}
Given an embedding $\iota\colon K\hookrightarrow \bbC$, there exists a positive real number $\rho(\mathrm{f};\iota)>0$ (consult Notation \ref{not:endomorphisms}) for which the Dirichlet series $Z(\mathrm{f};q^{-s})$ converges (absolutely) in the half-plane $\mathrm{Re}(s)>\frac{\mathrm{log}(\rho(\mathrm{f};\iota))}{\mathrm{log}(q)}$.
\end{proposition}
\begin{proof}
Consider the (auxiliar) formal power series $Z'(\mathrm{f};t):=\sum_{n\geq 1} \mathrm{tr}(\mathrm{f}^{\circ n})\frac{t^n}{n}$. Note that, by definition, we have $Z(\mathrm{f};t) = \mathrm{exp}(Z'(\mathrm{f};t))$. It follows from the classical theory of Dirichlet series that the abscissa of convergence of the Dirichlet series $Z(\mathrm{f};q^{-s})$ is less than or equal to the abscissa of convergence of the (auxiliar) Dirichlet series $Z'(\mathrm{f};q^{-s})$. Moreover, the later abscissa of convergence is equal to $-\frac{\mathrm{log}(\mathrm{rad}(Z'(\mathrm{f};t)))}{\mathrm{log}(q)}$, where $\mathrm{rad}(Z'(\mathrm{f};t))$ stands for the radius of convergence of $Z'(\mathrm{f};t)$. Therefore, in order to conclude the proof, it suffices to show that 
\begin{equation}\label{eq:inequality}
-\frac{\mathrm{log}(\mathrm{rad}(Z'(\mathrm{f};t)))}{\mathrm{log}(q)}\leq \frac{\mathrm{log}(\rho(\mathrm{f};\iota))}{\mathrm{log}(q)}\,.
\end{equation}
In this regards, we have the following (in)equalities: 
\begin{eqnarray}
\frac{1}{\mathrm{rad}(Z'(\mathrm{f};t))} & = & \mathrm{limsup}_{n\to \infty} |\frac{\mathrm{tr}(\mathrm{f}^{\circ n})}{n}|^{\frac{1}{n}} \label{eq:star11} \\
& \leq & \mathrm{limsup}_{n\to \infty} |\mathrm{tr}(\mathrm{f}^{\circ n})|^{\frac{1}{n}} \nonumber \\
& \leq & \mathrm{limsup}_{n\to \infty} (d^{\frac{1}{n}} \cdot \rho(\mathrm{f};\iota)) \label{eq:star22} \\
& = & \rho(\mathrm{f};\iota)\,, \nonumber
\end{eqnarray}
where \eqref{eq:star11} is the Cauchy-Hadamard theorem and \eqref{eq:star22} follows from Lemma \ref{lem:bound} (similarly to the proof of Lemma \ref{lem:bound}, $d:=\mathrm{dim}\,TP_+(N\!\!M)_{1/p} + \mathrm{dim}\,TP_-(N\!\!M)_{1/p}$). This clearly yields the searched inequality \eqref{eq:inequality}.
\end{proof}
By combining Proposition \ref{prop:convergence} with Theorem \ref{thm:zeta}(i), we observe that the Hasse-Weil zeta function $\zeta(\mathrm{f};s)$ admits a (unique) meromorphic continuation to the entire complex plane. Moreover, by definition, we have $\zeta(\mathrm{f}; s + \frac{2\pi i}{\mathrm{log}(q)})=\zeta(\mathrm{f};s)$.
\begin{example}[Classical Hasse-Weil zeta functions]\label{ex:Hasse}
Let $N\!\!M$ and $\mathrm{f}$ (and $X$) be as in Example \ref{ex:classical}. Thanks to the natural isomorphism \eqref{eq:Scholze} and to Deligne's proof \cite{Deligne} of the analogue of the Riemann hypothesis (consult also Katz-Messing \cite{KM}), we have $\rho(\mathrm{f};\iota)=q^{\mathrm{dim}(X)}$ (independently of the embedding $\iota$). Moreover, it is well-know that we have an equality of formal power series $Z_X(t)= \prod_{x \in X} \frac{1}{1- t^{\mathrm{deg}(x)}}$. Consequently, in this particular case, the meromorphic function \eqref{eq:function} reduces to the classical Hasse-Weil zeta function $\zeta_X(s):= \prod_{x \in X} \frac{1}{1- (q^{\mathrm{deg}(x)})^{-s}}$ (with $\mathrm{Re}(s)>\mathrm{dim}(X)$).
\end{example}
\begin{example}[Classical Artin $L$-functions]\label{ex:Artin}
Let $N\!\!M$ and $\mathrm{f}$ (and $G$, $X$ and $\chi$) be as in Example \ref{ex:L-functions}. Similarly to Example \ref{ex:Hasse}, we have $\rho(\mathrm{f};\iota)\leq q^{\mathrm{dim}(X)}$ (independently of the embedding $\iota$). Moreover, it is well-known that we have an equality of formal power series $L_{X,G,\chi}(t)= \prod_{y \in X/G} \frac{1}{\mathrm{det}(\id - (\frac{1}{|I_x|}\sum_{g \in D_x, g \mapsto \mathrm{fr}_x}\mathrm{rep}_\chi(g))t^{\mathrm{deg}(y)})}$, where $x \in X$ is a(ny) pre-image of $y$, $D_x \subset G$ is the stabilizer of $x$, $I_x$ is the kernel of the canonical surjection $D_x \twoheadrightarrow \mathrm{Gal}(\kappa(x)/\kappa(y))$, $\mathrm{fr}_x$ is the Frobenius element of the Galois group $\mathrm{Gal}(\kappa(x)/\kappa(y))$, and $\mathrm{rep}_\chi$ is the complex representation associated to the character $\chi$; consult Serre \cite[\S2.2]{Serre} for further details. Consequently, in this particular case, the meromorphic function \eqref{eq:function} reduces to the classical (Emil) Artin $L$-function (with $\mathrm{Re}(s)>\mathrm{dim}(X)$):
$$
\cL_{X, G, \chi}(s) := \prod_{y \in X/G} \frac{1}{\mathrm{det}(\id - (\frac{1}{|I_x|}\sum_{g \in D_x, g \mapsto \mathrm{fr}_x}\mathrm{rep}_\chi(g))(q^{\mathrm{deg}(y)})^{-s})} \,.
$$
\end{example}
\begin{example}[Orbifolds]
Let $N\!\!M$ and $\mathrm{f}$ (and $G$ and $X$) be as in Example \ref{ex:orbifolds}. In this particular case, the meromorphic function \eqref{eq:function} reduces to the product of (Emil) Artin $L$-functions $\prod_{g \in G/\!\sim} \cL_{X^g, C(g)}(s)$.
\end{example}
\subsection*{Functional equation}
Note that Theorem \ref{thm:zeta}(ii) yields the functional equation:
\begin{corollary}\label{cor:functional}
When $\mathrm{f}$ is invertible, we have the following functional equation between meromorphic functions $ \zeta((\mathrm{f}^{-1})^\vee; -s) = -(q^{-s})^{\mathrm{tr}(\id_{N\!\!M})}\cdot \mathrm{det}(\mathrm{f}) \cdot \zeta(\mathrm{f};s)$.
\end{corollary} 
Intuitively speaking, Corollary \ref{cor:functional} describes a ``symmetry'' between the Hasse-Weil zeta functions $\zeta(\mathrm{f};s)$ and $\zeta((\mathrm{f}^{-1})^\vee;s)$ along the vertical line $\mathrm{Re}(s)=0$.
\subsection*{Cohomological interpretation}
Given an embedding $\iota\colon K \hookrightarrow \bbC$, consider the induced endomorphism of finite-dimensional $\bbZ/2$-graded $\bbC$-vector spaces:
$$ \theta_\pm(\mathrm{f};\iota):=TP_\pm(\mathrm{f})_{1/p} \otimes_{K, \iota} \bbC \colon TP_\pm(N\!\!M)_{1/p} \otimes_{K, \iota} \bbC \too TP_\pm(N\!\!M)_{1/p} \otimes_{K, \iota} \bbC\,.$$
Note that the above equality of formal power series \eqref{eq:equation-main} yields the following cohomological interpretation of the Hasse-Weil zeta functions of endomorphisms:
\begin{corollary}[Cohomological interpretation I]\label{cor:holomorphic}
Given an embedding $\iota \colon K \hookrightarrow \bbC$, we have the following equality of meromorphic functions:
\begin{equation}\label{eq:zeta2}
\zeta(\mathrm{f}; s) = \frac{\mathrm{det}(\id - q^{-s} \theta_-(\mathrm{f};\iota)\,|\, TP_-(N\!\!M)_{1/p} \otimes_{K, \iota} \bbC)}{\mathrm{det}(\id - q^{-s}\theta_+(\mathrm{f};\iota)\,|\, TP_+(N\!\!M)_{1/p} \otimes_{K, \iota} \bbC)}\,.
\end{equation}
\end{corollary}
In what follows, we assume that $\mathrm{f}\colon N\!\!M \to N\!\!M$ is invertible. Recall from the multiplicative Jordan decomposition that $\theta_+(\mathrm{f};\iota)= \theta_{+,s}(\mathrm{f};\iota) \circ \theta_{+, u}(\mathrm{f};\iota)$, where $\theta_{+,s}(\mathrm{f};\iota)$, resp. $\theta_{+, u}(\mathrm{f};\iota)$, is a semi-simple, resp. unipotent, matrix with complex coefficients. Let $\Theta_+(\mathrm{f};\iota):=\mathrm{log}_q(\theta_{+,s}(\mathrm{f};\iota)) + \mathrm{log}_q(\theta_{+,u}(\mathrm{f};\iota))$, where $\mathrm{log}_q(\theta_{+,s}(\mathrm{f};\iota))$ stands for the diagonal matrix obtained by applying the logarithmic function $\mathrm{log}_q(-)$ (with respect to the principal branch $]-\frac{\pi}{\mathrm{log}(q)}, \frac{\pi}{\mathrm{log}(q)}]$) to $\theta_{+,s}(\mathrm{f};\iota)$, and $\mathrm{log}_q(\theta_{+,u}(\mathrm{f};\iota))$ stands for the matrix obtained by applying the convergent series $\frac{1}{\mathrm{log}(q)}\mathrm{log}(-)$ to $\theta_{+,u}(\mathrm{f};\iota)$. Consider also the following infinite-dimensional $\bbC$-vector space 
$$TP_{\mathrm{even}}(N\!\!M)_{1/p} \otimes_{K, \iota} \bbC:=\bigoplus_{j \in \bbZ} (TP_+(N\!\!M)_{1/p} \otimes_{K, \iota} \bbC)$$ 
and its endomorphism $\Theta_{\mathrm{even}}(\mathrm{f};\iota)$ defined as $v \mapsto \Theta_+(\mathrm{f};\iota)(v) + j \cdot \frac{2\pi i}{\mathrm{log}(q)} \cdot v$ for every homogeneous vector $v$ of degree $j$. In the same vein, consider the infinite-dimensional $\bbC$-vector space $TP_{\mathrm{odd}}(N\!\!M)_{1/p} \otimes_{K, \iota} \bbC$ and its endomorphism $\Theta_{\mathrm{odd}}(\mathrm{f};\iota)$.

By combining Corollary \ref{cor:holomorphic} with Deninger's theory of regularized determinants $\mathrm{det}_\infty(-)$ (consult (the proof of) \cite[Cor.~2.8]{Deninger}), we obtain moreover the following cohomological interpretation of the Hasse-Weil zeta functions of endomorphisms:
\begin{corollary}[Cohomological interpretation II]\label{cor:II}
Given an embedding $\iota \colon K \hookrightarrow \bbC$, we have the following equality of meromorphic functions:
\begin{equation}\label{eq:zeta3}
\zeta(\mathrm{f}; s) = \frac{\mathrm{det}_\infty(\id\cdot s - \Theta_{\mathrm{odd}}(\mathrm{f};\iota)\,|\, TP_{\mathrm{odd}}(N\!\!M)_{1/p} \otimes_{K, \iota} \bbC)}{\mathrm{det}_\infty(\id\cdot s - \Theta_{\mathrm{even}}(\mathrm{f};\iota)\,|\,TP_{\mathrm{even}}(N\!\!M)_{1/p} \otimes_{K, \iota} \bbC)}\,.
\end{equation}
\end{corollary}
In the particular case of the classical Hasse-Weil zeta function $\zeta_X(s)$ of a smooth proper $k$-scheme $X$ (see Example \ref{ex:Hasse}), the cohomological interpretations \eqref{eq:zeta2} and \eqref{eq:zeta3} were originally established by Hesselholt in \cite{Hesselholt}. On the one hand, his proof of \eqref{eq:zeta2} was based on the combination of certain spectral sequences (the conjugate and the Hodge spectral sequences) with Berthelot's cohomological interpretation \cite[\S VIII]{Berthelot} of the classical Hasse-Weil zeta function in terms of crystalline cohomology theory. On the other hand, his proof of \eqref{eq:zeta3} was (similarly) obtained from \eqref{eq:zeta2} as an application of Deninger's theory of regularized determinants. Corollaries \ref{cor:holomorphic} and \ref{cor:II} provide us with an alternative proof of Hesselholt's cohomological interpretations of the classical Hasse-Weil zeta function. Moreover, these corollaries establish a far-reaching noncommutative generalization of Hesselholt's~results.

\section{Numerical Grothendieck groups}\label{sec:app2}
Given a proper dg category $\cA$, its Grothendieck group $K_0(\cA):=K_0(\cD_c(\cA))$ comes equipped with the following Euler bilinear pairing:
$$
\chi \colon  K_0(\cA) \times K_0(\cA) \too \bbZ \quad ([M],[N]) \mapsto \sum_n (-1)^n \mathrm{dim}\, \Hom_{\cD_c(\cA)}(M,N[n])\,.
$$
This bilinear pairing is, in general, not symmetric neither skew-symmetric. Nevertheless, when $\cA$ is moreover smooth, the associated left and right kernels of $\chi$ agree; see \cite[Prop.~4.24]{book}. Consequently, under these assumptions on $\cA$, we have a well-defined {\em numerical Grothendieck group} $K_0(\cA)/_{\!\!\sim \mathrm{num}}:=K_0(\cA)/\mathrm{Ker}(\chi)$.

Our third main result is the following:
\begin{theorem}\label{thm:main2}
$K_0(\cA)/_{\!\!\sim \mathrm{num}}$ is a finitely generated free abelian group.
\end{theorem}
To the best of the author's knowledge, Theorem \ref{thm:main2} is new in the literature. It was claimed (without proof) by Kuznetsov in his work \cite{Kuznetsov} on surface-like categories.

Given a smooth proper $k$-scheme $X$, let us write $\cZ^\ast(X)/_{\!\!\sim \mathrm{num}}$ for the (graded) group of algebraic cycles on $X$ up to numerical equivalence. 
\begin{corollary}\label{cor:main2}
$\cZ^\ast(X)/_{\!\!\sim \mathrm{num}}$ is a finitely generated free abelian (graded) group.
\end{corollary}
\subsection*{Proof of Theorem \ref{thm:main2}}
Note first that, by definition, the numerical Grothendieck group $K_0(\cA)/_{\!\!\sim \mathrm{num}}$ is torsion-free. Therefore, in order to prove that $K_0(\cA)/_{\!\!\sim \mathrm{num}}$ is a finitely generated free abelian group, it suffices to show that the $\bbQ$-vector space $K_0(\cA)/_{\!\!\sim \mathrm{num}}\otimes_\bbZ \bbQ$ is finite-dimensional. Consider the following bilinear pairing:
\begin{eqnarray*}
\varphi\colon K_0(\cA) \times K_0(\cA^\op) \too K_0(k) \simeq \bbZ && ([M],[P]) \mapsto [M\otimes^{\bf L}_\cA P]\,.
\end{eqnarray*}
Under the isomorphisms \eqref{eq:Homs}, the bilinear pairing $\varphi$ corresponds to the following composition pairing (we removed the subscripts $\NChow(k)$):
$$
\Hom(U(k),U(\cA))\times \Hom(U(\cA),U(k))\too \mathrm{End}(U(k)) \quad (\mathrm{f},\mathrm{g}) \mapsto \mathrm{g}\circ \mathrm{f} = \mathrm{tr}(\mathrm{g}\circ \mathrm{f})\,.
$$
Let $M\in \cD_c(\cA)$. Thanks to Proposition \ref{prop:key2} below, we have $\chi([M],[N])=0$ for every $N \in \cD_c(\cA)$ if and only if $\varphi([M],[P])=0$ for every $P \in \cD_c(\cA^\op)$. This implies that $K_0(\cA)/_{\!\!\sim \mathrm{num}}$ is isomorphic to $\Hom_{\NNum(k)}(U(k),U(\cA))$. Consequently, we conclude that the $\bbQ$-vector space $K_0(\cA)/_{\!\!\sim \mathrm{num}} \otimes_\bbZ \bbQ$ is isomorphic to
\begin{equation}\label{eq:vector}
\Hom_{\NNum(k)}(U(k),U(\cA))\otimes_\bbZ \bbQ \simeq \Hom_{\NNum(k)_\bbQ}(U(k)_\bbQ, U(\cA)_\bbQ)\,,
\end{equation}
where the latter isomorphism follows from the compatibility of the $\otimes$-ideal $\cN$ with the change-of-coefficients along $\bbZ \to \bbQ$; consult \cite[Prop.~1.4.1 3)]{Brug}. Finally, since the category $\NNum(k)_\bbQ$ is abelian semi-simple (see Theorem \ref{thm:main}), the $\bbQ$-vector space \eqref{eq:vector} is finite-dimensional. This concludes the proof.
\begin{proposition}\label{prop:key2}
Given a right dg $\cA$-module $M \in \cD_c(\cA)$, we have $\chi([M],[N])=0$ for every $N \in \cD_c(\cA)$ if and only if $\varphi([M],[P])=0$ for every $P \in \cD_c(\cA^\op)$.
\end{proposition}
\begin{proof}
Note first that the differential graded structure of the dg category of complexes of $k$-vector spaces $\cC_\dg(k)$ makes the category of right dg $\cA$-modules $\cC(\cA)$ into a dg category $\cC_\dg(\cA)$. Given right dg $\cA$-modules $M,N \in \cD_c(\cA)$, let us write ${\bf R}\Hom_\cA(M,N) \in \cD_c(k)$ for the corresponding (derived) complex of $k$-vector spaces. Since the Grothendieck class $[{\bf R}\Hom_\cA(M,N)] \in K_0(k)\simeq \bbZ$ agrees with the alternating sum $\sum_n (-1)^n \mathrm{dim}\, \Hom_{\cD_c(\cA)}(M,N[n])$, the Euler bilinear pairing $\chi\colon K_0(\cA) \times K_0(\cA) \to \bbZ$ can be re-written as $([M],[N]) \mapsto [{\bf R}\Hom_\cA(M,N)]$.

Given a right dg $\cA$-module $M$, let us denote by $M^\ast$ the right dg $\cA^\op$-module obtained by composing the dg functor $M\colon \cA^\op \to \cC_\dg(k)$ with the $k$-linear duality $(-)^\ast$. Note that since the dg category $\cA$ is proper, the assignment $M\mapsto M^\ast$ gives rise to an equivalence of categories $\cD_c(\cA) \simeq \cD_c(\cA^\op)^\op$. 

We now have the ingredients necessary to conclude the proof. Assume that $\chi([M],[N])=0$ for every $N \in \cD_c(\cA)$. We need to show that $\varphi([M],[P])=0$ for every $P \in \cD_c(\cA^\op)$. In other words, we need to show that the Grothendieck class $[M\otimes^{\bf L}_\cA P] \in K_0(k)\simeq \bbZ$ is zero for every $P \in \cD_c(\cA^\op)$. Thanks to Lemma \ref{lem:key}(ii) below, the (derived) complex of $k$-vector spaces $M\otimes^{\bf L}_\cA P$ is $k$-linear dual to ${\bf R}\Hom_\cA(M,P^\ast)$. Consequently, since $[{\bf R}\Hom_\cA(M,P^\ast)]=\chi([M],[P^\ast])=0$, we conclude that $[M\otimes^{\bf L}_\cA P]=0$. Let us assume now that $\varphi([M],[P])=0$ for every $P \in \cD_c(\cA^\op)$. We need to show that $\chi([M],[N])=0$ for every $N \in \cD_c(\cA)$. In other words, we need to show that the Grothendieck class $[{\bf R}\Hom_\cA(M,N)]\in K_0(k)\simeq \bbZ$ is zero for every $N \in \cD_c(\cA^\op)$. Thanks to Lemma \ref{lem:key}(i) below, the (derived) complex of $k$-vector spaces ${\bf R}\Hom_\cA(M,N)$ is $k$-linear dual to $M\otimes^{\bf L}N^\ast$. Consequently, since $[M\otimes^{\bf L}_\cA N^\ast]=\varphi([M],[N^\ast])=0$, we conclude that $[{\bf R}\Hom_\cA(M,N)]=0$.
\end{proof}
\begin{lemma}\label{lem:key}
\begin{itemize}
\item[(i)] Given $M,N \in \cD_c(\cA)$, the associated (derived) complexes of $k$-vector spaces ${\bf R}\Hom_\cA(M,N)$ and $M\otimes^{\bf L}_\cA N^\ast$ are $k$-linear dual to each other. 
\item[(ii)] Given $M\in \cD_c(\cA)$ and $P \in \cD_c(\cA^\op)$, the associated (derived) complexes of $k$-vector spaces ${\bf R}\Hom_\cA(M,P^\ast)$ and $M\otimes^{\bf L}_\cA P$ are $k$-linear dual to each other. 
\end{itemize}
\end{lemma}
\begin{proof}
We prove solely item (i); the proof of item (ii) is similar. Given an object $y \in \cA$ and an integer $n \in \bbZ$, consider the following right dg $\cA$-module:
\begin{eqnarray*}
\widehat{y}[n] \colon \cA^\op \too \cC_\dg(k) && x \mapsto \cA(x,y)[n]\,.
\end{eqnarray*}
Every right dg $\cA$-module $M \in \cD_c(\cA)$ is a retract of a finite homotopy colimit of right dg $\cA$-modules of the form $\widehat{y}[n]$. When $M=\widehat{y}[n]$, we have:
$$ {\bf R}\Hom_\cA(\widehat{y}[n], N)^\ast \simeq (N(y)[-n])^\ast \simeq N^\ast(y)[n]\simeq \widehat{y}[n]\otimes^{\bf L}_\cA N^\ast\,.$$
In the same vein, when $M=\mathrm{hocolim}_{i\in I} \,\widehat{y_i}[n_i]$, we have natural isomorphisms
\begin{eqnarray}
{\bf R}\Hom_\cA(\mathrm{hocolim}_{i\in I} \,\widehat{y_i}[n_i], N)^\ast & \simeq & (\mathrm{holim}_{i\in I}  \,{\bf R}\Hom_\cA(\widehat{y_i}[n_i],N))^\ast \nonumber \\
& \simeq & \mathrm{hocolim}_{i \in I} \,{\bf R}\Hom_\cA(\widehat{y_i}[n_i],N)^\ast \label{eq:star-last}\\
& \simeq & \mathrm{hocolim}_{i \in I} \,(\widehat{y_i}[n_i]\otimes^{\bf L}_\cA N^\ast)  \nonumber\\
& \simeq & (\mathrm{hocolim}_{i \in I}\, \widehat{y_i}[n_i])\otimes^{\bf L}_\cA N^\ast\,, \nonumber
\end{eqnarray}
where in \eqref{eq:star-last} we use the fact that the indexing category $I$ is finite. The proof follows now from the fact that the $k$-linear duality $(-)^\ast$ preserves retracts.
\end{proof}
\begin{remark}[Kontsevich's noncommutative numerical motives]\label{rk:Kontsevich}
Let $F$ be a field of coefficients of characteristic zero. Recall from \S\ref{sub:numerical} that the symmetric monoidal category $\NNum(k)_F$ is rigid. Therefore, as the above proof of Theorem \ref{thm:main2} shows (with $\bbQ$ replaced by $F$), we have a natural isomorphism
$$\Hom_{\NNum(k)_F}(U(\cA)_F,U(\cB)_F) \simeq K_0(\cA^\op \otimes \cB)_F/\mathrm{Ker}(\chi\otimes_\bbZ F)$$ 
for any two smooth proper dg categories $\cA$ and $\cB$. This implies that we can alternatively define the category $\NNum(k)_F$ as the idempotent completion of the quotient of $\NChow(k)_F$ by the $\otimes$-ideal $\mathrm{Ker}(\chi\otimes_\bbZ F)$. This was the approach used by Kontsevich in his seminal talk \cite{IAS}.
\end{remark}
\subsection*{Proof of Corollary \ref{cor:main2}}
Thanks to the Hirzebruch-Riemann-Roch theorem, the Chern character induces an isomorphism of $\bbQ$-vector spaces:
\begin{equation}\label{eq:2-sides}
K_0(\perf_\dg(X))/_{\!\!\sim \mathrm{num}} \otimes_\bbZ \bbQ\stackrel{\simeq}{\too} \cZ^\ast(X)/_{\!\!\sim \mathrm{num}}\otimes_\bbZ \bbQ\,.
\end{equation}
Hence, by applying Theorem \ref{thm:main2} to the dg category $\cA=\perf_\dg(X)$, we conclude that the right-hand side of \eqref{eq:2-sides} is a finite-dimensional $\bbQ$-vector space. The proof follows now from the fact that $\cZ^\ast(X)/_{\!\!\sim \mathrm{num}}$ is a torsion-free abelian (graded) group.
\begin{remark}
The analogue of Theorem \ref{thm:main2} (and of Corollary \ref{cor:main2}), with $k$ of characteristic zero, was proved in \cite[Thm.~1.2]{Separable}.
\end{remark}
\section{Noncommutative motivic Galois (super-)groups}\label{app:3}
In this section $F$ denotes a field of coefficients of characteristic zero. We assume moreover that $F \subseteq K$, resp. $K \subseteq F$.
\subsection*{Noncommutative standard conjectures of type $C^+$ and $D$}
In this subsection we assume that $k$ is perfect. Given a smooth proper $k$-scheme $X$, let us write $\pi_X^i$ for the $i^{\mathrm{th}}$ K\"unneth projector of the crystalline cohomology $H^\ast_{\mathrm{crys}}(X)$ of $X$, $Z^\ast(X)_F$ for the $F$-vector space of algebraic cycles on $X$, and $Z^\ast(X)_F/_{\!\sim \mathrm{hom}}$ and $Z^\ast(X)_F/_{\!\sim \mathrm{num}}$ for the quotients of $Z^\ast(X)_F$ with respect to the homological and numerical equivalence relations, respectively. Following Grothendieck \cite{Grothendieck-conjectures} (see also Kleiman \cite{Kleim1, Kleim}), the standard conjecture of type $C^+$, denoted by $C^+(X)$, asserts that the K\"unneth projectors $\pi^+_X:=\sum_{i\,\mathrm{even}} \pi^i_X$ and $\pi^{-}_X:=\sum_{i\,\mathrm{odd}} \pi^i_X$ are algebraic, and the standard conjecture of type $D$, denoted by $D(X)$, asserts that $Z^\ast(X)_F/_{\!\sim \mathrm{hom}}=Z^\ast(X)_F/_{\!\sim \mathrm{num}}$. Both these conjectures hold when $\mathrm{dim}(X)\leq 2$. Moreover, the standard conjecture of type $C^+$ holds for abelian varieties (see Kleiman \cite[2. Appendix]{Kleim}) and also when $k$ is finite (see Katz-Messing \cite{KM}). Besides these cases (and some other cases scattered in the literature), the aforementioned important conjectures remain wide open.

Given a smooth proper dg category $\cA$, consider the K\"unneth projectors $\pi^\cA_+$ and $\pi^\cA_-$ of the $\bbZ/2$-graded $K$-vector space $TP_\pm(\cA)_{1/p}$ (when $F \subseteq K$), resp. the K\"unneth projectors $\pi^\cA_+$ and $\pi^\cA_-$ of the $\bbZ/2$-graded $F$-vector space $TP_\pm(\cA)_{1/p}\otimes_K F$ (when $K \subseteq F$). Consider also the following $F$-vector spaces\footnote{Consult \S\ref{sec:app2} for an alternative definition of $K_0(\cA)_F/_{\!\sim \mathrm{num}}$.}:
\begin{eqnarray*}
K_0(\cA)_F/_{\!\sim \mathrm{hom}}:= \Hom_{\NHom(k)_F}(U(k)_F, U(\cA)_F)\\ 
K_0(\cA)_F/_{\!\sim \mathrm{num}}:= \Hom_{\NNum(k)_F}(U(k)_F, U(\cA)_F)\,.
\end{eqnarray*}
Under these notations, Grothendieck's standard conjectures of type $C^+$ and $D$ admit the following noncommutative counterparts:

\vspace{0.1cm}

{\bf Conjecture $C^+_{\mathrm{nc}}(\cA)$:} {\it The K\"unneth projectors $\pi_+^\cA$ and $\pi_-^\cA$ are {\em algebraic}, \ie they belong to the $F$-vector subspace $\End_{\NHom(k)_F}(U(\cA)_F)$.}

\vspace{0.1cm}

{\bf Conjecture $D_{\mathrm{nc}}(\cA)$:} {\it The equality $K_0(\cA)_F/_{\!\sim\mathrm{hom}}=K_0(\cA)_F/_{\!\sim\mathrm{num}}$ holds.}

\vspace{0.1cm}

As proved in \cite[Thm.~1.1]{positiveCD}, given a smooth proper $k$-scheme $X$, we have the equivalences of conjectures $C^+(X) \Leftrightarrow C^+_{\mathrm{nc}}(\perf_\dg(X))$ and $D(X)\Leftrightarrow D_{\mathrm{nc}}(\perf_\dg(X))$. Intuitively speaking, this shows us that Grothendieck's standard conjectures of type $C^+$ and $D$ belong not only to the realm of algebraic geometry but also to the broad setting of smooth proper dg categories.
\subsection*{(Super-)Tannakian categories}
Recall from \S\ref{sub:numerical} the definition of the category of noncommutative numerical motives $\NNum(k)_F$. Note that this category is {\em not} Tannakian because the Euler characteristic of its objects is often negative. For example, given a smooth projective curve $C$ of genus $g$, we have $\mathrm{tr}(\id_{N\!\!M})=2-2g$ where $N\!\!M:=U(\perf_\dg(C))_F$.
\begin{theorem}\label{thm:Tannakian}
\begin{itemize}
\item[(i)] Assume that $C^+_{\mathrm{nc}}(\cA)$ holds for every smooth proper dg category $\cA$. Under this assumption, $\NNum(k)_F$ can be modified into a Tannakian category $\NNum^\dagger(k)_F$. Moreover, if $D_{\mathrm{nc}}(\cA)$ also holds for every smooth proper dg category $\cA$, then $\NNum^\dagger(k)_F$ comes equipped with an explicit fiber functor given by topological periodic cyclic homology.
\item[(ii)] The category $\NNum(k)_F$ is super-Tannakian in the sense of Deligne \cite{Deligne-Moscow}. If $F$ is algebraically closed, then $\NNum(k)_F$ is moreover neutral super-Tannakian.
\end{itemize}
\end{theorem}
\begin{proof}
We address solely the case $F \subseteq K$; the proof of the other case is similar. 

The proof of item (i) is similar to the proof of \cite[Thms.~1.4 and 1.6]{JEMS} (with $k$ of characteristic zero): simply replace the periodic cyclic homology functor (consult \cite[Thm.~9.2]{JEMS}) by the functor \eqref{eq:TP1} and run the same proof.

We now prove item (ii). By construction, the category $\NNum(k)_F$ is $F$-linear, additive, and rigid symmetric monoidal. Thanks to Theorem \ref{thm:main}, it is moreover abelian semi-simple. Hence, making use of Deligne's characterization of (neutral) super-Tannakian categories (see \cite[Thm.~0.6]{Deligne-Moscow}), it suffices to show that the category $\NNum(k)_F$ is Schur-finite. In the case where $k$ is perfect, the proof that the category $\NNum(k)_F$ is Schur-finite is similar to the proof of \cite[Prop.~8.1]{JEMS}: simply replace the periodic cyclic homology functor (consult \cite[Thm.~9.2]{JEMS}) by the functor \eqref{eq:TP1} and run the same proof. Let us assume now that $k$ is arbitrary. As explained in \cite[\S4.10.2]{book}, the functor \eqref{eq:closure} descends to noncommutative numerical motives $-\otimes_k k^{\mathrm{perf}}\colon \NNum(k)_F \to \NNum(k^{\mathrm{perf}})_F$. Moreover, since the field extension $k \subseteq k^{\mathrm{perf}}$ is purely inseparable and $F$ is of characteristic zero, the preceding functor is fully-faithful; see \cite[Prop.~4.11]{Rigidity}. Therefore, using the fact that the field $k^{\mathrm{perf}}$ is perfect, we conclude from above that the category $\NNum(k)_F$ is Schur-finite.
\end{proof}
\begin{corollary}\label{cor:super}
\begin{itemize}
\item[(i)] Assume that $C^+(X)$ holds for every smooth proper $k$-scheme $X$. Under this assumption, $\Num(k)_F$ can be modified into a Tannakian category $\Num^\dagger(k)_F$. Moreover, if $D(X)$ also holds for every smooth proper $k$-scheme $X$, then $\Num^\dagger(k)_F$ comes equipped with an explicit fiber functor given by crystalline cohomology.
\item[(ii)] The category $\Num(k)_F$ is super-Tannakian in the sense of Deligne \cite{Deligne-Moscow}. If $F$ is algebraically closed, then $\Num(k)_F$ is moreover neutral super-Tannakian.
\end{itemize}
\end{corollary}
\begin{proof}
We address solely the case $F \subseteq K$; the proof of the other case is similar. 

The proof of item (i) is similar to the proof of Theorem \ref{thm:Tannakian}(i): simply replace the conjectures $C^+_{\mathrm{nc}}(\cA)$ and $D_{\mathrm{nc}}(\cA)$ by the conjectures $C^+(X)$ and $D(X)$, the categories $\NChow(k)_F$ and $\NNum(k)_F$ by the categories $\Chow(k)_F$ and $\Num(k)_F$, and the functor \eqref{eq:TP1} by the following composition:
\begin{equation}\label{eq:composition-Chow}
\Chow(k)_F \stackrel{\gamma}{\too} \Chow(k)_F/_{\!\!-\otimes F(1)} \stackrel{\Phi}{\too} \NChow(k)_F \stackrel{\eqref{eq:TP1}}{\too} \mathrm{vect}_{\bbZ/2}(K)\,.
\end{equation}
Note that thanks to the natural isomorphism \eqref{eq:Scholze}, the composition \eqref{eq:composition-Chow} corresponds to the functor that sends a Chow motive $\mathfrak{h}(X)_F$ to the finite-dimensional $\bbZ/2$-graded $K$-vector space $(\bigoplus_{i\,\mathrm{even}}H^i_{\mathrm{crys}}(X), \bigoplus_{i\,\mathrm{odd}}H^i_{\mathrm{crys}}(X))$.

We now prove item (ii). By construction, the category $\Num(k)_F$ is $F$-linear, additive, and rigid symmetric monoidal. Thanks to Corollary \ref{cor:main}, it is moreover abelian semi-simple. Therefore, similarly to the proof of Theorem \ref{thm:Tannakian}(ii), it suffices to show that $\Num(k)_F$ is Schur-finite. As explained in \cite[\S4.6]{book}, the functor $\Phi$ (consult diagram \eqref{eq:bridge}) descends to the categories of numerical motives. Therefore, we can consider the composition $\Phi\circ \gamma\colon \Num(k)_F \to \Num(k)_F/_{\!\!- \otimes F(1)} \to \NNum(k)_F$. Since the $F$-linear symmetric monoidal functor $\gamma$, resp. $\Phi$, is faithful, resp. fully-faithful, and the category $\NNum(k)_F$ is Schur-finite, we hence conclude that the category $\Num(k)_F$ is also Schur-finite.
\end{proof}
\subsection*{Noncommutative motivic Galois (super-)groups}
Assume that $C^+_{\mathrm{nc}}(\cA)$ and $D_{\mathrm{nc}}(\cA)$ hold for every smooth proper dg category $\cA$. By combining Theorem \ref{thm:Tannakian}(i) with the Tannakian formalism applied to the fiber functor (when $F\subseteq K$)
\begin{equation}\label{eq:fiber1}
\NNum^\dagger(k)_F \stackrel{TP_\pm(-)_{1/p}}{\too} \mathrm{vect}_{\bbZ/2}(K) \stackrel{\mathrm{forget}}{\too} \mathrm{vect}(K)\,,
\end{equation} 
resp. to the fiber functor (when $K\subseteq F$)
\begin{equation}\label{eq:fiber2}
\NNum^\dagger(k)_F \stackrel{TP_\pm(-)_{1/p}\otimes_K F}{\too} \mathrm{vect}_{\bbZ/2}(F) \stackrel{\mathrm{forget}}{\too} \mathrm{vect}(F)\,,
\end{equation} 
we obtain an affine group $K$-scheme, resp. affine group $F$-scheme, $\mathrm{Gal}(\NNum^\dagger(k)_F)$ called the {\em noncommutative motivic Galois group}. Since the Tannakian category $\NNum^\dagger(k)_F$ is not only abelian but moreover semi-simple, the noncommutative motivic Galois group is moreover {\em pro-reductive}, \ie its unipotent radical is trivial. 

Similarly, since $C^+_{\mathrm{nc}}(\perf_\dg(X)) \Rightarrow C^+(X)$ and $D_{\mathrm{nc}}(\perf_\dg(X)) \Rightarrow D(X)$ for every smooth proper $k$-scheme $X$, by combining Corollary \ref{cor:super}(i) with the Tannakian formalism applied to the fiber functor (when $F \subseteq K$)
\begin{equation}\label{eq:fiber11} \Num^\dagger(k)_F \stackrel{\gamma}{\too} \NNum^\dagger(k)_F/_{\!\!-\otimes F(1)} \stackrel{\Phi}{\too} \NNum^\dagger(k)_F \stackrel{\eqref{eq:fiber1}}{\too} \mathrm{vect}(K)\,,
\end{equation}
resp. to the fiber functor (when $K \subseteq F$)
\begin{equation}\label{eq:fiber22}
\Num^\dagger(k)_F \stackrel{\gamma}{\too} \NNum^\dagger(k)_F/_{\!\!-\otimes F(1)} \stackrel{\Phi}{\too} \NNum^\dagger(k)_F \stackrel{\eqref{eq:fiber2}}{\too} \mathrm{vect}(F)\,,
\end{equation}
we obtain an affine group $K$-scheme, resp. affine group $F$-scheme, $\mathrm{Gal}(\Num^\dagger(k)_F)$ called the {\em motivic Galois group}. Note that thanks to the natural isomorphism \eqref{eq:Scholze}, the composition \eqref{eq:fiber11}, resp. \eqref{eq:fiber22},  corresponds to the functor that sends a numerical motive $\mathfrak{h}(X)_F$ to the finite-dimensional $K$-vector space $\bigoplus_i H^i_{\mathrm{crys}}(X)$, resp. to the finite-dimensional $F$-vector space $\bigoplus_i H^i_{\mathrm{crys}}(X)\otimes_K F$. Finally, since the Tannakian category $\Num^\dagger(k)_F$ is not only abelian but moreover semi-simple, the noncommutative motivic Galois group is moreover pro-reductive.

Assume now that $F$ is algebraically closed. By combining Theorem \ref{thm:Tannakian}(ii) with Deligne's super-Tannakian formalism \cite{Deligne-Moscow}, applied to a super-fiber functor 
\begin{equation}\label{eq:super-fiber}
 \NNum(k)_F \too \mathrm{vect}_{\bbZ/2}(F)\,,
 \end{equation} 
we obtain an affine super-group $F$-scheme $\mathrm{sGal}(\NNum(k)_F)$ called the {\em noncommutative motivic Galois super-group}; in the case where $F$ is not algebraically closed, the noncommutative motivic Galois super-group is defined only over a (very large) commutative $F$-algebra. Similarly, by combining Corollary \ref{cor:super}(ii) with Deligne's super-Tannakian formalism, applied to the super-fiber functor
\begin{equation}\label{eq:fiber-last}
\Num(k)_F \stackrel{\gamma}{\too} \Num(k)_F/_{\!\!-\otimes F(1)} \stackrel{\Phi}{\too} \NNum(k)_F \stackrel{\eqref{eq:super-fiber}}{\too} \mathrm{vect}_{\bbZ/2}(F)\,,
\end{equation}
we obtain an affine super-group $F$-scheme $\mathrm{sGal}(\Num(k)_F)$ called the {\em motivic Galois super-group}. Our fourth main result is the following:
\begin{theorem}\label{thm:Galois}
\begin{itemize}
\item[(i)] Assume that $C^+_{\mathrm{nc}}(\cA)$ and $D_{\mathrm{nc}}(\cA)$ hold for every smooth proper dg category $\cA$. Under these assumptions, the following composition of functors $\Phi\circ \gamma \colon \Num^\dagger(k)_F \to \Num^\dagger(k)_F/_{\!\!-\otimes F(1)} \to \NNum^\dagger(k)_F$ induces a faithfully flat morphism of affine group $K$-schemes (when $F \subseteq K$), resp. affine group $F$-schemes (when $K \subseteq F$),
$$ \mathrm{Gal}(\NNum^\dagger(k)_F) \stackrel{(\Phi \circ \gamma)^\ast}{\too} \mathrm{Kernel} (\mathrm{Gal}(\Num^\dagger(k)_F) \stackrel{t^\ast}{\too} \bbG_m)\,,$$
where $\bbG_m$ stands for the multiplicative (super-)group scheme and $t$ for the inclusion of the category of Tate motives into the category of numerical motives.
\item[(ii)]  Assume that $F$ is algebraically closed. Under this assumption, the following composition of functors $\Phi\circ \gamma\colon \Num(k)_F \to \Num(k)_F/_{\!\!-\otimes F(1)} \to \NNum(k)_F$ induces a faithfully flat morphism of affine super-group $F$-schemes:
$$ \mathrm{sGal}(\NNum(k)_F) \stackrel{(\Phi \circ \gamma)^\ast}{\too} \mathrm{Kernel} (\mathrm{sGal}(\Num(k)_F) \stackrel{t^\ast}{\too} \bbG_m)\,.$$
\end{itemize}
\end{theorem}
Theorem \ref{thm:Galois} (and Theorem \ref{thm:Tannakian}) was envisioned by Kontsevich in his seminal talk \cite{IAS}. Intuitively speaking, it shows us that the ``$\otimes$-symmetries'' of the commutative world which can be lifted to the noncommutative world are precisely those which become trivial when restricted to Tate motives.
\subsection*{Proof of Theorem \ref{thm:Galois}}
We start with a general result of independent interest. Let $(\mathrm{D}, \otimes, {\bf 1})$ be an $F$-linear, abelian semi-simple, rigid symmetric monoidal category equipped with a $\otimes$-invertible object $\cO \in \mathrm{D}$. Consider the composition
\begin{equation}\label{eq:composition-aux}
\langle \cO \rangle \stackrel{t}{\hookrightarrow} \mathrm{D} \stackrel{\gamma}{\too} \mathrm{Idem}(\mathrm{D}/_{\!\!-\otimes \cO})\,,
\end{equation}
where $\mathrm{Idem}(\mathrm{D}/_{\!\!-\otimes \cO})$ stands for the idempotent completion of the orbit category $\mathrm{D}/_{\!\!-\otimes \cO}$ and $\langle \cO \rangle$ stands for the smallest $F$-linear, abelian, rigid symmetric monoidal subcategory of $\mathrm{D}$ containing the $\otimes$-invertible object $\cO$.
\begin{proposition}\label{prop:aux}
Assume that $\Hom_{\mathrm{D}}(\cO^{\otimes i}, \cO^{\otimes j})\simeq \delta_{ij} \cdot F$ and that the category $\mathrm{Idem}(\mathrm{D}/_{\!\!-\otimes \cO})$ is abelian semi-simple.
\begin{itemize}
\item[(i)] Given a $F$-linear, faithful, symmetric monoidal functor $\omega\colon \mathrm{Idem}(\mathrm{D}/_{\!\!-\otimes \cO}) \to \mathrm{vect}(K)$ (when $F \subseteq K$), resp. a $F$-linear, faithful, symmetric monoidal functor $\omega\colon \mathrm{Idem}(\mathrm{D}/_{\!\!-\otimes \cO})\to\mathrm{vect}(F)$ (when $K \subseteq F$), the composition \eqref{eq:composition-aux} induces the short exact sequence of affine group $K$-schemes, resp. $F$-schemes,
\begin{equation}\label{eq:SES1}
1 \too \mathrm{Aut}^\otimes(\omega) \stackrel{\gamma^\ast}{\too} \mathrm{Aut}^\otimes(\omega \circ \gamma) \stackrel{t^\ast}{\too} \bbG_m \too 1\,.
\end{equation}
\item[(ii)] Assume moreover that $F$ is algebraically closed. Given a $F$-linear, faithful, symmetric monoidal functor $\omega\colon \mathrm{Idem}(\mathrm{D}/_{\!\!-\otimes \cO}) \to \mathrm{vect}_{\bbZ/2}(F)$, the composition \eqref{eq:composition-aux} induces the short exact sequence of affine super-group $F$-schemes:
\begin{equation}\label{eq:SES}
1 \too \mathrm{Aut}^\otimes(\omega) \stackrel{\gamma^\ast}{\too} \mathrm{Aut}^\otimes(\omega \circ \gamma) \stackrel{t^\ast}{\too} \bbG_m \too 1\,.
\end{equation}
\end{itemize}
\end{proposition}
\begin{proof}
We address solely the case $F \subseteq K$; the proof of the other case is similar. We start by proving item (i). We claim first that the following affine group $K$-scheme $\mathrm{Aut}^\otimes(\omega \circ \gamma\circ t)$ is isomorphic to $\bbG_m$. Let $\mathrm{vect}_\bbZ(F)$ be the category of finite-dimensional $\bbZ$-graded $F$-vector spaces. Since $\Hom_{\mathrm{D}}(\cO^{\otimes i}, \cO^{\otimes j})\simeq \delta_{ij} \cdot F$, note that we have the following equivalence of categories:
\begin{eqnarray}\label{eq:equivalence-1}
\mathrm{vect}_\bbZ(F) \stackrel{\simeq}{\too} \langle \cO\rangle && \{V_n\}_{n \in \bbZ} \mapsto \bigoplus_{n \in \bbZ} (\cO^{\otimes n})^{\oplus (\mathrm{dim}\,V_n)}\,.
\end{eqnarray}
Moreover, since $(\gamma \circ t)(\cO)\simeq {\bf 1}$ in the category $\mathrm{Idem}(\mathrm{D}/_{\!\!-\otimes \cO})$, the fiber functor $\omega \circ \gamma \circ t \colon \langle \cO \rangle  \to \mathrm{vect}(K)$ corresponds to the following composition
\begin{equation}\label{eq:composition-last}
\mathrm{vect}_\bbZ(F) \stackrel{-\otimes_F K}{\too} \mathrm{vect}_\bbZ(K) \stackrel{\mathrm{forget}}{\too} \mathrm{vect}(K)\,.
\end{equation}
This implies our claim because it is well-known that the affine group $K$-scheme of $\otimes$-automorphisms of \eqref{eq:composition-last} is isomorphic to $\bbG_m$. 

We now claim that \eqref{eq:SES1} is a short exact sequence. Since $(\gamma \circ t)(\cO) \simeq {\bf 1}$ in the category $\mathrm{Idem}(\mathrm{D}/_{\!\!-\otimes \cO})$, the composition $t^\ast \circ \gamma^\ast$ is trivial. Moreover, since the functor $t$ is fully-faithful, the induced morphism $t^\ast$ is faithfully flat. Furthermore, since every object of $\mathrm{Idem}(\mathrm{D}/_{\!\!-\otimes \cO})$ is a direct summand of an object in the image of $\gamma$, the induced morphism $\gamma^\ast$ is a closed immersion. Therefore, it remains only to show that $\mathrm{Ker}(t^\ast) \subseteq \mathrm{Im}(\gamma^\ast)$. Let $\eta \in \mathrm{Ker}(t^\ast)$. Concretely, $\eta$ consists of a $\otimes$-automorphism of the fiber functor $\omega \circ \gamma$ such that $\eta_{t(\cO)} = \id$. By construction of the orbit category $\mathrm{D}/_{\!\!-\otimes \cO}$ (consult \S\ref{sub:orbit}), the idempotent completion of $\eta$ yields then a $\otimes$-automorphism of the fiber functor $\omega$. This shows that $\eta$ also belongs to $\mathrm{Im}(\gamma^\ast)$.  

We now prove item (ii). As above, we have the equivalence of categories \eqref{eq:equivalence-1}. Moreover, since $(\gamma \circ t)(\cO)\simeq {\bf 1}$ in the category $\mathrm{Idem}(\mathrm{D}/_{\!\!-\otimes \cO})$, the super-fiber functor 
\begin{equation}\label{eq:super-fiber-functor}
\omega \circ \gamma\circ t \colon \langle \cO \rangle  \too \mathrm{vect}_{\bbZ/2}(F)
\end{equation} takes values in the full subcategory $\mathrm{vect}(F) \subset \mathrm{vect}_{\bbZ/2}(F)$ of evenly supported $\bbZ/2$-graded $F$-vector spaces. Consequently, \eqref{eq:super-fiber-functor} corresponds to the forgetful functor $\mathrm{vect}_\bbZ(F) \to \mathrm{vect}(F)$. This implies that $\mathrm{Aut}^\otimes(\omega \circ \gamma\circ t)$ is isomorphic to the affine (super-)group $F$-scheme $\bbG_m$. Finally, the proof of the short exact sequence \eqref{eq:SES} is similar to the above proof of the short exact sequence \eqref{eq:SES1}.
\end{proof}
We now have all the ingredients necessary to prove Theorem \ref{thm:Galois}. We start by proving item (i). Consider the $F$-linear, abelian semi-simple, rigid symmetric monoidal category $\mathrm{D}:=\Num^\dagger(k)_F$ equipped with the $\otimes$-invertible object $\cO:=F(1)$. The smallest $F$-linear, abelian, rigid symmetric monoidal subcategory of $\Num^\dagger(k)_F$ containing the Tate motive $F(1)$ agrees with the category of Tate motives. Hence, by construction, we have $\Hom_{\Num^\dagger(k)_F}(F(1)^{\otimes i}, F(1)^{\otimes j})\simeq \delta_{ij} \cdot F$. Moreover, since the category $\NNum^\dagger(k)_F$ is abelian semi-simple and the functor $\Phi\colon \mathrm{Idem}(\Num^\dagger(k)_F/_{\!\!-\otimes F(1)}) \to \NNum^\dagger(k)_F$ is fully-faithful, the (idempotent complete) category $\mathrm{Idem}(\Num^\dagger(k)_F/_{\!\!-\otimes F(1)})$ is also semi-simple; see \cite[Lem.~2]{Jannsen}. Therefore, by applying the above Proposition \ref{prop:aux}(i) to the $F$-linear, faithful, symmetric monoidal functor (when $F\subseteq K$)
\begin{equation*}\label{eq:aux1}
\omega\colon \mathrm{Idem}(\Num^\dagger(k)_F/_{\!\!-\otimes F(1)}) \stackrel{\Phi}{\too} \NNum^\dagger(k)_F \stackrel{\eqref{eq:fiber1}}{\too} \mathrm{vect}(K)\,,
\end{equation*}
resp. to the $F$-linear, faithful, symmetric monoidal functor (when $K \subseteq F$)
\begin{equation*}\label{eq:aux2}
\omega\colon \mathrm{Idem}(\Num^\dagger(k)_F/_{\!\!-\otimes F(1)}) \stackrel{\Phi}{\too} \NNum^\dagger(k)_F \stackrel{\eqref{eq:fiber2}}{\too} \mathrm{vect}(F)\,,
\end{equation*}
we obtain the short exact sequence of affine group $K$-schemes, resp. $F$-schemes,
\begin{equation}\label{eq:1}
1 \too \mathrm{Gal}(\mathrm{Idem}(\Num^\dagger(k)_F/_{\!\!-\otimes F(1)})) \stackrel{\gamma^\ast}{\too} \mathrm{Gal}(\Num^\dagger(k)_F) \stackrel{t^\ast}{\too} \bbG_m \too 1\,.
\end{equation}
Using the fact that the functor $\Phi$ is fully-faithful, we have moreover the following faithfully flat morphism of affine group $K$-schemes, resp. $F$-schemes,
\begin{equation}\label{eq:2}
\mathrm{Gal}(\NNum^\dagger(k)_F) \stackrel{\Phi^\ast}{\too} \mathrm{Gal}(\mathrm{Idem}(\Num^\dagger(k)_F/_{\!\!-\otimes F(1)}))\,.
\end{equation} 
Hence, the proof of item (i) follows now from the combination of \eqref{eq:1} with \eqref{eq:2}.

The proof of item (ii) is similar to the above proof of item (i): replace the category $\Num^\dagger(k)_F$ by the category $\Num(k)_F$, the functors \eqref{eq:fiber1} and \eqref{eq:fiber2} by the functor \eqref{eq:super-fiber}, and Proposition \ref{prop:aux}(i) by Proposition \ref{prop:aux}(ii).

\begin{remark}
The analogue of Theorem \ref{thm:Galois}, with $k$ of characteristic zero, was proved in \cite[Thm.~1.7]{JEMS}. Therein, we made essential use of Deligne-Milne's theory of Tate-triples. In contrast, the above proof of Theorem \ref{thm:Galois} completely avoids the use of Deligne-Milne's theory of Tate-triples. Moreover, this simpler proof can also be applied to the case where $k$ is of characteristic zero.
\end{remark}
\section{Motivic measures}
In this section we assume that $k$ is perfect. Let us write $\mathrm{Var}(k)$ for the category of {\em varieties}, \ie reduced separated $k$-schemes of finite type. Recall that the {\em Grothendieck ring of varieties} $K_0\mathrm{Var}(k)$ is defined as the quotient of the free abelian group on the set of isomorphism classes of varieties $[X]$ by the relations $[X]=[Z] + [X\backslash Z]$, where $Z$ is a closed subvariety of $X$. Multiplication is induced by the product of varieties. Although very important, the structure of $K_0 \mathrm{Var}(k)$ remains poorly understood. In order to capture some of its flavor, some {\em motivic measures}, \ie ring homomorphisms $\mu\colon K_0 \mathrm{Var}(k) \to R$, have been built.
\begin{example}
When $k=\bbF_q$ is a finite field, the assignment $[X] \mapsto \#X(\bbF_q)$ gives rise to a motivic measure $\mu_\#\colon K_0 \mathrm{Var}(k) \to \bbZ$.
\end{example}
\begin{example}
The assignment $[X] \mapsto \sum_i (-1)^i \mathrm{dim}\, H^i_{\mathrm{rig}, \mathrm{c}}(X)$, where $H_{\mathrm{rig}, \mathrm{c}}^\ast(X)$ stands for Berthelot's compactly supported rigid cohomology (consult \cite{Berthelot-rigid}), gives rise to a motivic measure $\mu_{\mathrm{rig}}\colon K_0\mathrm{Var}(k) \to \bbZ$.
\end{example}
Let $K_0(\NChow(k)_\bbQ)$ be the Grothendieck ring of the $\bbQ$-linear, additive, rigid symmetric monoidal category of noncommutative Chow motives $\NChow(k)_\bbQ$.
\begin{proposition}\label{prop:existence}
There exists a motivic measure $\mu_{\mathrm{nc}}\colon K_0 \mathrm{Var}(k) \to K_0(\NChow(k)_\bbQ)$ such that $\mu_{\mathrm{nc}}([X])= [U(\perf_\dg(X))_\bbQ]$ for every smooth proper $k$-scheme $X$.
\end{proposition}
\begin{proof}
Let $\mathrm{Var}(k)^p$ be the category whose objects are the varieties and whose morphisms are the proper maps. Following \cite[Lem.~5.5.6 and Prop.~ 5.5.8]{Kelly}, there exists a symmetric monoidal functor $M^c(-)_\bbQ \colon \mathrm{Var}(k)^p \to \mathrm{DM}_{\mathrm{gm}}(k)_\bbQ$ with values in Voevodsky's triangulated category of geometric motives. Moreover, following \cite[Prop.~5.5.5]{Kelly}, given a variety $X$ and a closed subvariety $Z \subset X$, we have the following distinguished triangle in $\mathrm{DM}_{\mathrm{gm}}(k)_\bbQ$:
$$ M^c(Z)_\bbQ \too M^c(X)_\bbQ \too M^c(X\backslash Z)_\bbQ \too M^c(Z)_\bbQ[1]\,.$$
 Consequently, we obtain the following motivic measure:
\begin{eqnarray}\label{eq:ring1}
K_0 \mathrm{Var}(k) \too K_0(\mathrm{DM}_{\mathrm{gm}}(k)_\bbQ) && [X] \mapsto [M^c(X)_\bbQ]\,.
\end{eqnarray} 
Recall from the proof of Corollary \ref{cor:main} that the category of Chow motives comes equipped with a (contravariant) functor $\mathfrak{h}(-)_\bbQ\colon \mathrm{SmProp}(k)^\op \to \Chow(k)_\bbQ$ defined on smooth proper $k$-schemes. As proved in \cite[Prop.~2.1.4]{Voevodsky}, there exists a $\bbQ$-linear, fully-faithful, symmetric monoidal functor $\Psi\colon \Chow(k)_\bbQ \to \mathrm{DM}_{\mathrm{gm}}(k)_\bbQ$ such that $\Psi(\mathfrak{h}(X)) \simeq M^c(X)_\bbQ$ for every smooth proper $k$-scheme $X$. Moreover, thanks to \cite[Prop.~2.3.3]{Bondarko}, $\Psi$ induces a ring isomorphism $K_0(\Chow(k)_\bbQ)\simeq K_0(\mathrm{DM}_{\mathrm{gm}}(k)_\bbQ)$. Hence, by combining the above considerations with the commutative diagram \eqref{eq:bridge} (with $F=\bbQ$), we obtain the following motivic measure:
$$ \mu_{\mathrm{nc}}\colon K_0\mathrm{Var}(k) \stackrel{\eqref{eq:ring1}}{\too} K_0(\mathrm{DM}_{\mathrm{gm}}(k)_\bbQ) \simeq K_0(\Chow(k)_\bbQ) \stackrel{\Phi \circ \gamma}{\too} K_0(\NChow(k)_\bbQ)\,.$$
Finally, note that, by construction, we have $\mu_{\mathrm{nc}}([X])= [U(\perf_\dg(X))_\bbQ]$ for every smooth proper $k$-scheme $X$. This finishes the proof.
\end{proof}
Our fifth main result is the following:
\begin{theorem}\label{thm:measure}
The following holds:
\begin{itemize}
\item[(i)] The motivic measure $\mu_\#$ does {\em not} factors through $\mu_{\mathrm{nc}}$.
\item[(ii)] The motivic measure $\mu_{\mathrm{rig}}$ factors through $\mu_{\mathrm{nc}}$. 
\end{itemize}
\end{theorem}
Roughly speaking, Theorem \ref{thm:measure} shows us that in contrast with the ``counting-of-points'' procedure, the Euler characteristic of Berthelot's rigid cohomology can still be recovered from the theory of noncommutative Chow motives.
\begin{proof}
We start by proving item (i). Let $\bbP^n$ be the $n^{\mathrm{th}}$ projective space for some integer $n>1$. On the one hand, following \cite{Beilinson}, the category $\perf(\bbP^n)$ admits a full exceptional collection of length $n+1$. As explained in the proof of Theorem \ref{thm:TP}, this implies that $U(\perf_\dg(\bbP^n))_\bbQ\simeq U(k)_\bbQ^{\oplus n+1}$. Consequently, $[U(\perf_\dg(\bbP^n))_\bbQ]=n+1$ in the Grothendieck group $K_0(\NChow(k)_\bbQ)$. On the other hand, since $k=\bbF_q$, we have $\# \bbP^n(\bbF_q)=1+q +q^2 + \cdots + q^n$. Since $n+1 \neq 1 + q + q^2 + \cdots + q^n$, this shows that the motivic measure $\mu_\#$ does {\em not} factors through $\mu_{\mathrm{nc}}$.

We now prove item (ii). Note that since the functor \eqref{eq:TP1} (with $F=\bbQ$) is $\bbQ$-linear and symmetric monoidal, it induces a ring homomorphism:
\begin{equation}\label{eq:induced1}
TP_\pm(-)_{1/p}\colon K_0(\NChow(k)_\bbQ) \too K_0(\mathrm{vect}_{\bbZ/2}(K)) \,.
\end{equation}
Moreover, we have the following computation:
\begin{eqnarray*}\label{eq:iso}
K_0(\mathrm{vect}_{\bbZ/2}(K))\simeq \bbZ[\epsilon]/(\epsilon^2=1) && [(V_+, V_-)]\mapsto \mathrm{dim}\, V_+ + \mathrm{dim}\, V_- \cdot \epsilon\,.
\end{eqnarray*}
Therefore, we can consider the following composition:
\begin{equation}\label{eq:induced2}
K_0(\NChow(k)_\bbQ) \stackrel{\eqref{eq:induced1}}{\too} \bbZ[\epsilon]/(\epsilon^2=1) \stackrel{\epsilon\mapsto -1}{\too} \bbZ\,.
\end{equation}
We claim that the motivic measure $\mu_{\mathrm{rig}}$ agrees with the composition of $\mu_{\mathrm{nc}}$ with \eqref{eq:induced2}. Since $\mu_{\mathrm{rig}}$ factors through the motivic measure \eqref{eq:ring1} and, as explained in the proof of Proposition \ref{prop:existence}, we have an isomorphism $K_0(\Chow(k)_\bbQ)\simeq K_0(\mathrm{DM}_{\mathrm{gm}}(k)_\bbQ)$, it suffices to consider the case where $X$ is a smooth proper $k$-scheme. On the one hand, since $X$ is smooth and proper, we have 
\begin{equation}\label{eq:equalities}
\mu_{\mathrm{rig}}([X]):= \sum_i (-1)^i \mathrm{dim}\, H^i_{\mathrm{rig}, \mathrm{c}}(X) = \sum_i (-1)^i \mathrm{dim}\, H^i_{\mathrm{crys}}(X)\,,
\end{equation}
where $H^\ast_{\mathrm{crys}}(X)$ stands for crystalline cohomology. On the other hand, thanks to the natural isomorphism \eqref{eq:Scholze}, the image of $[U(\perf_\dg(X))_\bbQ]$ under the above composition \eqref{eq:induced2} also agrees with $\sum_i (-1)^i \mathrm{dim}\, H^i_{\mathrm{crys}}(X)$.
\end{proof}
\begin{remark}[Loss of information]
Intuitively speaking, the above proof of item (i) shows us that in the particular case of a finite field $k=\bbF_q$ the passage from the commutative world (where $\# \bbP^n(\bbF_q)= 1 + q + q^2 + \cdots + q^n$) to the noncommutative world (where $[U(\perf_\dg(\bbP^n))_\bbQ]=n+1$) may be understood as setting $q=1$. However, note that if we keep track of the Frobenius endomorphism (as done in \S\ref{app:1}-\S\ref{sec:app-2}), then there is no loss~of information when passing from the commutative world~to~the~noncommutative~world.
\end{remark}

\medbreak\noindent\textbf{Acknowledgments:} The author is grateful to Lars Hesselholt for useful discussions concerning topological periodic cyclic homology and for comments on a previous version of Corollaries \ref{cor:holomorphic} and \ref{cor:II}. The author is also grateful to Maxim Kontsevich for comments on a previous version of Theorem \ref{thm:zeta} and for pointing out Remark \ref{rk:zero}. The author also would like to thank the anonymous referee for his/her comments, and to the Mittag-Leffler Institute and to the Hausdorff Research Institute for Mathematics for their hospitality.

\end{document}

\end{proof}